\documentclass[a4paper]{article}
\usepackage{relsize}
\usepackage{amsmath}
\usepackage{amsthm}
\usepackage{amsfonts}
\usepackage{mathtools}
\usepackage{graphicx}
\usepackage{multicol}
\usepackage{version}
\newtheorem{theorem}{Theorem}
\newtheorem{proposition}{Proposition}
\newcommand{\pointfont}[1]{\mathit{#1}}
\newcommand{\ptA}{\pointfont{A}}
\newcommand{\ptB}{\pointfont{B}}
\newcommand{\ptC}{\pointfont{C}}
\newcommand{\ptD}{\pointfont{D}}
\newcommand{\ptE}{\pointfont{E}}
\newcommand{\ptF}{\pointfont{F}}
\newcommand{\ptDA}{\ptD_{\!\ptA}}
\newcommand{\ptEA}{\ptE_{\!\ptA}}
\newcommand{\ptFA}{\ptF_{\!\ptA}}
\newcommand{\rA}{{r_{\!\ptA}}}
\newcommand{\rB}{{r_{\ptB}}}
\newcommand{\rC}{{r_{\ptC}}}
\newcommand{\alphaA}{{\alpha_{\!\ptA}}}
\newcommand{\alphaB}{{\alpha_{\ptB}}}
\newcommand{\alphaC}{{\alpha_{\ptC}}}
\newcommand{\betaA}{{\beta_{\!\ptA}}}
\newcommand{\betaB}{{\beta_{\ptB}}}
\newcommand{\betaC}{{\beta_{\ptC}}}
\newcommand{\gammaA}{{\gamma_{\!\ptA}}}
\newcommand{\gammaB}{{\gamma_{\ptB}}}
\newcommand{\gammaC}{{\gamma_{\ptC}}}
\newcommand{\sigmaA}{{\sigma_{\!\ptA}}}
\newcommand{\sigmaB}{{\sigma_{\ptB}}}
\newcommand{\sigmaC}{{\sigma_{\ptC}}}
\title{Schellbach-style Formulae for the Derousseau-Pampuch Generalizations of the Malfatti Circles}
\author{Hiroyasu Kamo\thanks{Nara Women's University, Nara, Japan; email: \texttt{wd@ics.nara-wu.ac.jp}}}
\begin{document}
\maketitle
\begin{abstract}
It is known that
there exist 32 triplets of circles
such that each circle is tangent to the other two circles
and to two of the sides of the triangle or their extensions.
We provide formulae
to obtain the radii of the circles
for each of the 32 triplets
from the side lengths of the reference triangle
by means of trigonometric or hyperbolic functions.
\end{abstract}
\section{Introduction}
The configuration of
three circles inside a triangle such that each circle is tangent to
the other two circles and to two of the sides of the triangle
has been studied
for more than two centuries.
Today,
such three circles are called the \emph{Malfatti circles}
of the triangle.
\begin{figure}[htbp]
\begin{center}
\includegraphics{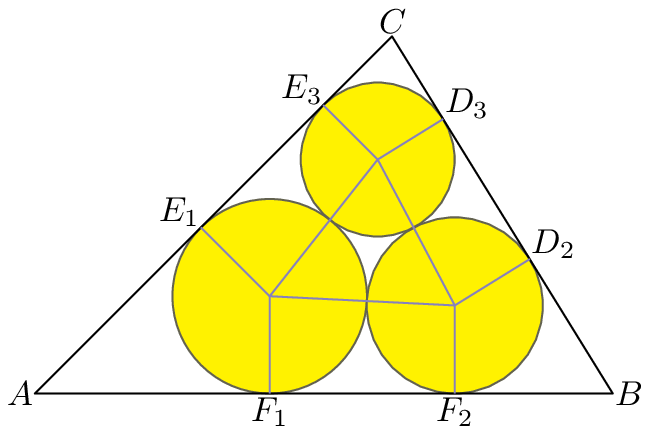}
\caption{}
\end{center}
\end{figure}

Sometime before 1773,
Naonobu Ajima (1732?--1798),
who was a samurai, or a member of the military class in old Japan,
found a method to calculate the diameters of the Malfatti circles
from the side lengths of an arbitrary triangle.
The method was called
\textit{Nanzan-shi san-sha naiy\=o san-en jutsu}
(``\mbox{Nanzan's} method on a triangle that includes three circles'',
as Nanzan is a pen name of Ajima's)
or \textit{San-sha san-en jutsu} in short.
A brief description of the method
is found in \cite[I \P14]{fukyu-sanpo}.
A detailed description of
the method including a proof is found in \cite{sansha-sanen}.
Unfortunately, Ajima's method
as well as any other results by Japanese mathematicians in those days
was inaccessible from outside Japan
until the Edo shogunate, the former government of Japan (1603--1868),
abandoned the isolation policy in 1854.

In 1803,
an Italian mathematician Gianfrancesco Malfatti (1731--1807)
\cite{Malfatti1803} gave
a construction to draw the Malfatti circles
for an arbitrary triangle.
Despite Malfatti's unawareness of Ajima's works,
Malfatti's construction is considered identical in many parts
to Ajima's method.


In 1852,
Schellbach \cite{Schellbach1853a}\cite{Schellbach1853b} gave a set of formulae
to obtain the distances between the vertices and the tangent points of the circles on the sides
from the side lengths of an arbitrary triangle
by using trigonometric functions.
The same formulae with a proof essentially identical to Schellbach's
are described in English in \cite[\S30]{Dorrie1965}\cite{Edmunds1881}.

In 1895,
Derousseau \cite{Derousseau1895} generalized the Malfatti circles
by removing the condition
that the three circles are inside the triangle.
Derousseau proved that
there exist 32 triplets of circles
such that each circle is tangent to
the other two circles
and to two of the sides of the reference triangle or their extensions.
Some alternative proofs of the existence are known.
In 1904,
Pampuch \cite{Pampuch1904} gave another proof.
In 1930, Lob and Richmond \cite{Lob&Richmond1930}
gave yet another proof.

In this article,
we provide formulae
to obtain the radii of the circles
for each of the 32 triplets
from the side lengths
by means of trigonometric or hyperbolic functions.
In other words,
we provide Schellbach-style formulae
for all of the Derousseau-Pampuch generalizations.

\section{Notation}
Throughout this article, we use the following notation.

For a triangle $\ptA\ptB\ptC$,
let $a$, $b$, $c$ denote the lengths of the sides $\ptB\ptC$, $\ptC\ptA$, $\ptA\ptB$,
$s$ the semiperimeter,
$r$ the inradius,
and $\rA$, $\rB$, $\rC$ the exradii
as usual.

Let the incircle is tangent to the side $\ptB\ptC$ at $\ptD$,
\begin{figure}[htbp]
\begin{center}
\includegraphics{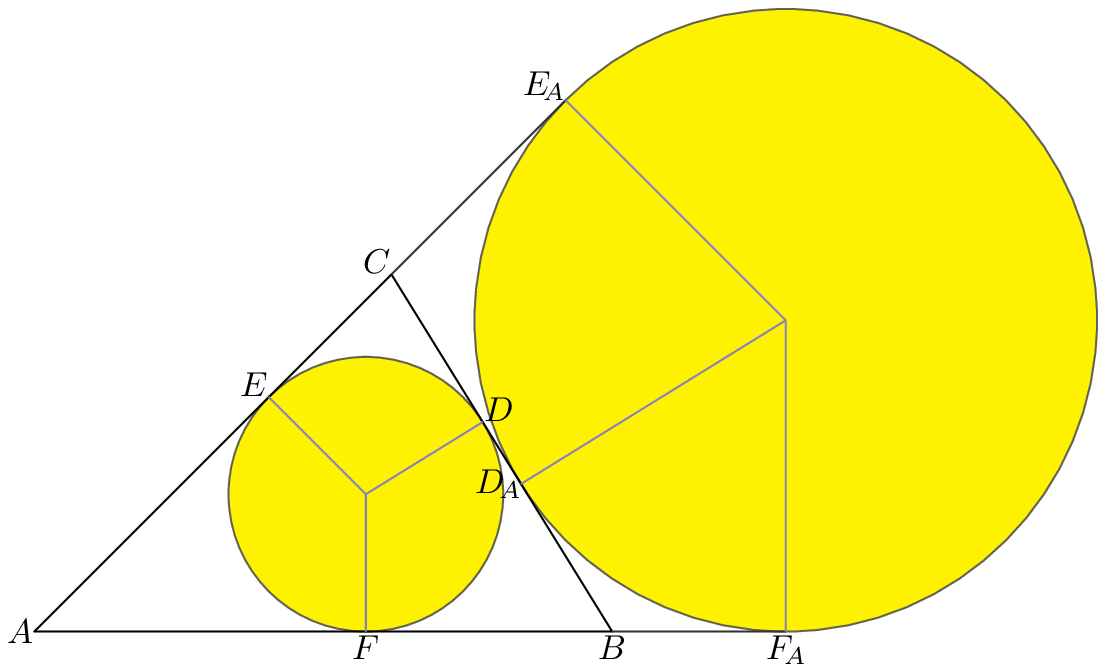}
\caption{}
\end{center}
\end{figure}
to the side $\ptC\ptA$ at $\ptE$,
and to the side $\ptA\ptB$ at $\ptF$.
Let the excircle corresponding to the vertex $\ptA$ is tangent
to the side $\ptB\ptC$ at $\ptDA$,
to the extension of the side $\ptA\ptC$ at $\ptEA$,
and to the extension of the side $\ptA\ptB$ at $\ptFA$.

Suppose the circle $\ptA'(r_1)$ is tangent to the line $\ptC\ptA$ at $\ptE_1$
and to the line $\ptA\ptB$ at $\ptF_1$,
the circle $\ptB'(r_2)$ is tangent to the line $\ptA\ptB$ at $\ptF_2$
and to the line $\ptB\ptC$ at $\ptD_2$,
the circle $\ptC'(r_3)$ is tangent to the line $\ptB\ptC$ at $\ptD_3$
and to the line $\ptC\ptA$ at $\ptE_3$,
and the three circles are tangent to one another.
Suppose the nine tangent points are distinct.

\section{Classification}
\label{sec:classify}
Since the center $\ptA'$ does not locate neither on the line $\ptA\ptB$ nor on the line $\ptA\ptC$,
it locates
inside $\angle\ptC\ptA\ptB$,
inside $\angle\overline{\ptC}\ptA\ptB$,
inside $\angle\ptC\ptA\overline{\ptB}$ or
inside $\angle\overline{\ptC}\ptA\overline{\ptB}$
where
an overline indicates that the angle has, as one of its sides,
the opposite ray instead of the ray including the triangle side.
For example, $\angle\ptC\ptA\overline{\ptB}$ denotes
the angle with the ray $\overrightarrow{\ptA\ptC}$
and the ray opposite to the ray $\overrightarrow{\ptB}$.
And $\angle\overline{\ptC}\ptA\overline{\ptB}$ denotes
the vertical angle of $\angle\ptC\ptA\ptB$.
Analogously,
the center $\ptB'$ locates
inside $\angle\ptA\ptB\ptC$,
inside $\angle\overline{\ptA}\ptB\ptC$,
inside $\angle\ptA\ptB\overline{\ptC}$ or
inside $\angle\overline{\ptA}\ptB\overline{\ptC}$
and
the center $\ptC'$ locates
inside $\angle\ptB\ptC\ptA$,
inside $\angle\overline{\ptB}\ptC\ptA$,
inside $\angle\ptB\ptC\overline{\ptA}$ or
inside $\angle\overline{\ptB}\ptC\overline{\ptA}$.

If the circles $\ptA'(r_1)$, $\ptB'(r_2)$, $\ptC'(r_3)$ lie
in $\varDelta_1$, $\varDelta_2$, $\varDelta_3$, respectively,
then ${\varDelta_1 \cap \varDelta_2} \not= \emptyset$,\quad
${\varDelta_1 \cap \varDelta_3} \not= \emptyset$,
and ${\varDelta_2 \cap \varDelta_3} \not= \emptyset$
since the three circles are tangent to one another.
Thus, for locations of the three centers $\ptA'$, $\ptB'$, $\ptC'$,
only $7$ out of the $64$ cases are consistent
to the condition that the circles are tangent to one another.
The following are the consistent cases.
\begin{quote}
\begin{tabular}{lccc}
& $\ptA'$ is inside & $\ptB'$ is inside & $\ptC'$ is inside
\\
Case~1
& $\angle\ptC\ptA\ptB$
& $\angle\ptA\ptB\ptC$
& $\angle\ptB\ptC\ptA$
\\
Case~2
& $\angle\ptC\ptA\ptB$
& $\angle\overline{\ptA}\ptB\ptC$
& $\angle\ptB\ptC\overline{\ptA}$
\\
Case~3
& $\angle\overline{\ptC}\ptA\overline{\ptB}$
& $\angle\ptA\ptB\overline{\ptC}$
& $\angle\overline{\ptB}\ptC\ptA$
\\
Case~4
& $\angle\ptC\ptA\overline{\ptB}$
& $\angle\ptA\ptB\ptC$
& $\angle\overline{\ptB}\ptC\ptA$
\\
Case~5
& $\angle\overline{\ptC}\ptA\ptB$
& $\angle\overline{\ptA}\ptB\overline{\ptC}$
& $\angle\ptB\ptC\overline{\ptA}$
\\
Case~6
& $\angle\overline{\ptC}\ptA\ptB$
& $\angle\ptA\ptB\overline{\ptC}$
& $\angle\ptB\ptC\ptA$
\\
Case~7
& $\angle\ptC\ptA\overline{\ptB}$
& $\angle\overline{\ptA}\ptB\ptC$
& $\angle\overline{\ptB}\ptC\overline{\ptA}$
\end{tabular}
\end{quote}

\section{Solution}
\subsection{Case 1}
\label{sec:case_1}

In Case~1, the following three conditions hold.
\begin{gather*}
\ptB\ptD_2+\ptD_3\ptC+\ptD_2\ptD_3 = \ptB\ptD + \ptD\ptC
\quad\text{or}\quad
\ptB\ptD_2+\ptD_3\ptC-\ptD_2\ptD_3 = \ptB\ptD + \ptD\ptC
,
\\
\ptA\ptE_1+\ptE_3\ptC+\ptE_1\ptE_3 = \ptA\ptE + \ptE\ptC
\quad\text{or}\quad
\ptA\ptE_1+\ptE_3\ptC-\ptE_1\ptE_3 = \ptA\ptE + \ptE\ptC
,
\\
\ptA\ptF_1+\ptF_2\ptB+\ptF_1\ptF_2 = \ptA\ptF + \ptF\ptB
\quad\text{or}\quad
\ptA\ptF_1+\ptF_2\ptB-\ptF_1\ptF_2 = \ptA\ptF + \ptF\ptB
.
\end{gather*}
By expressing the lengths by the radii and the angle sizes,
we obtain from the first disjunction that
\begin{gather}
r_2\cot\frac{B}{2} + r_3\cot\frac{C}{2} + 2\sqrt{r_2r_3} = r\cot\frac{B}{2}+r\cot\frac{C}{2}
\label{eq:in.BC+}
\shortintertext{or}
r_2\cot\frac{B}{2} + r_3\cot\frac{C}{2} - 2\sqrt{r_2r_3} = r\cot\frac{B}{2}+r\cot\frac{C}{2}
\label{eq:in.BC-}
,
\end{gather}
we obtain from the second disjunction that
\begin{gather}
r_1\cot\frac{A}{2} + r_3\cot\frac{C}{2} + 2\sqrt{r_1r_3} = r\cot\frac{A}{2}+r\cot\frac{C}{2}
\label{eq:in.AC+}
\shortintertext{or}
r_1\cot\frac{A}{2} + r_3\cot\frac{C}{2} - 2\sqrt{r_1r_3} = r\cot\frac{A}{2}+r\cot\frac{C}{2}
\label{eq:in.AC-}
,
\end{gather}
and we obtain from the second disjunction that
\begin{gather}
r_1\cot\frac{A}{2} + r_2\cot\frac{B}{2} + 2\sqrt{r_1r_2} = r\cot\frac{A}{2}+r\cot\frac{B}{2}
\label{eq:in.AB+}
\shortintertext{or}
r_1\cot\frac{A}{2} + r_2\cot\frac{B}{2} - 2\sqrt{r_1r_2} = r\cot\frac{A}{2}+r\cot\frac{B}{2}
\label{eq:in.AB-}
.
\end{gather}

Define $l$, $m$, $n$ by
\begin{align}
l &= \cot\frac{A}{2}
,&
m &= \cot\frac{B}{2}
,&
n &= \cot\frac{C}{2}
.
\label{eq:in.lmnABC}
\end{align}
Define $u$, $v$, $w$, $x$, $y$, $z$ by
\begin{gather*}
u =
\begin{dcases}
 \frac{\sqrt{r_2r_3}}{r} & \text{if \eqref{eq:in.BC+} holds,}
\\
-\frac{\sqrt{r_2r_3}}{r} & \text{if \eqref{eq:in.BC-} holds,}
\end{dcases}
\displaybreak[0]\\
v =
\begin{dcases}
 \frac{\sqrt{r_1r_3}}{r} & \text{if \eqref{eq:in.AC+} holds,}
\\
-\frac{\sqrt{r_1r_3}}{r} & \text{if \eqref{eq:in.AC-} holds,}
\end{dcases}
\displaybreak[0]\\
w =
\begin{dcases}
 \frac{\sqrt{r_1r_2}}{r} & \text{if \eqref{eq:in.AB+} holds,}
\\
-\frac{\sqrt{r_1r_2}}{r} & \text{if \eqref{eq:in.AB-} holds,}
\end{dcases}
\displaybreak[0]\\
x = \frac{r_1}{r},
\qquad
y = \frac{r_2}{r},
\qquad
z = \frac{r_3}{r}.
\end{gather*}
Then we have
\begin{equation}
\left\{
\begin{aligned}
my+nz+2u &= m+n,
\\
lx+nz+2v &= l+n,
\\
lx+my+2w &= l+m,
\\
xy &= w^2,
\\
xz &= v^2,
\\
yz &= u^2.
\end{aligned}
\right.
\label{eq:in}
\end{equation}

For any triangle $\ptA\ptB\ptC$,
if $l$, $m$, $n$ are defined by \eqref{eq:in.lmnABC},
then $lmn={l+m+n}$ holds.
On the other hand,
if positive reals $l$, $m$, $n$ satisfy $lmn={l+m+n}$,
then there exists a triangle $\ptA\ptB\ptC$
that satisfies \eqref{eq:in.lmnABC}.
Thus,
Case~1 can be reduced into solving the system of equations \eqref{eq:in}
for $u$, $v$, $w$, $x$, $y$, $z$
with positive real parameters $l$, $m$, $n$
under the restriction $lmn={l+m+n}$.

As we will show in Appendix~\ref{sec:how-to-solve-1},
the system of equations has the following $8$ solutions.
\begin{gather}
\left\{
\begin{aligned}
u&=\frac{\sqrt{l^2+1}-l+1}{2}
,\\
v&=\frac{\sqrt{m^2+1}-m+1}{2}
,\\
w&=\frac{\sqrt{n^2+1}-n+1}{2}
,\\
x&=\frac{l+m+n-1+\sqrt{l^2+1}-\sqrt{m^2+1}-\sqrt{n^2+1}}{2l}
,\\
y&=\frac{l+m+n-1-\sqrt{l^2+1}+\sqrt{m^2+1}-\sqrt{n^2+1}}{2m}
,\\
z&=\frac{l+m+n-1-\sqrt{l^2+1}-\sqrt{m^2+1}+\sqrt{n^2+1}}{2n}
.
\end{aligned}
\right.
\label{lmn:i1}
\displaybreak[0]\\
\left\{
\begin{aligned}
u&=\frac{\sqrt{l^2+1}-l+1}{2}
,\\
v&=-\frac{\sqrt{m^2+1}+m-1}{2}
,\\
w&=-\frac{\sqrt{n^2+1}+n-1}{2}
,\\
x&=\frac{l+m+n-1+\sqrt{l^2+1}+\sqrt{m^2+1}+\sqrt{n^2+1}}{2l}
,\\
y&=\frac{l+m+n-1-\sqrt{l^2+1}-\sqrt{m^2+1}+\sqrt{n^2+1}}{2m}
,\\
z&=\frac{l+m+n-1-\sqrt{l^2+1}+\sqrt{m^2+1}-\sqrt{n^2+1}}{2n}
.
\end{aligned}
\right.
\label{lmn:i2}
\displaybreak[0]\\
\left\{
\begin{aligned}
u&=-\frac{\sqrt{l^2+1}+l-1}{2}
,\\
v&=\frac{\sqrt{m^2+1}-m+1}{2}
,\\
w&=-\frac{\sqrt{n^2+1}+n-1}{2}
,\\
x&=\frac{l+m+n-1-\sqrt{l^2+1}-\sqrt{m^2+1}+\sqrt{n^2+1}}{2l}
,\\
y&=\frac{l+m+n-1+\sqrt{l^2+1}+\sqrt{m^2+1}+\sqrt{n^2+1}}{2m}
,\\
z&=\frac{l+m+n-1+\sqrt{l^2+1}-\sqrt{m^2+1}-\sqrt{n^2+1}}{2n}
.
\end{aligned}
\right.
\label{lmn:i3}
\displaybreak[0]\\
\left\{
\begin{aligned}
u&=-\frac{\sqrt{l^2+1}+l-1}{2}
,\\
v&=-\frac{\sqrt{m^2+1}+m-1}{2}
,\\
w&=\frac{\sqrt{n^2+1}-n+1}{2}
,\\
x&=\frac{l+m+n-1-\sqrt{l^2+1}+\sqrt{m^2+1}-\sqrt{n^2+1}}{2l}
,\\
y&=\frac{l+m+n-1+\sqrt{l^2+1}-\sqrt{m^2+1}-\sqrt{n^2+1}}{2m}
,\\
z&=\frac{l+m+n-1+\sqrt{l^2+1}+\sqrt{m^2+1}+\sqrt{n^2+1}}{2n}
.
\end{aligned}
\right.
\label{lmn:i4}
\displaybreak[0]\\
\left\{
\begin{aligned}
u&=-\frac{\sqrt{l^2+1}+l+1}{2}
,\\
v&=-\frac{\sqrt{m^2+1}+m+1}{2}
,\\
w&=-\frac{\sqrt{n^2+1}+n+1}{2}
,\\
x&=\frac{l+m+n+1-\sqrt{l^2+1}+\sqrt{m^2+1}+\sqrt{n^2+1}}{2l}
,\\
y&=\frac{l+m+n+1+\sqrt{l^2+1}-\sqrt{m^2+1}+\sqrt{n^2+1}}{2m}
,\\
z&=\frac{l+m+n+1+\sqrt{l^2+1}+\sqrt{m^2+1}-\sqrt{n^2+1}}{2n}
.
\end{aligned}
\right.
\label{lmn:i5}
\displaybreak[0]\\
\left\{
\begin{aligned}
u&=-\frac{\sqrt{l^2+1}+l+1}{2}
,\\
v&=\frac{\sqrt{m^2+1}-m-1}{2}
,\\
w&=\frac{\sqrt{n^2+1}-n-1}{2}
,\\
x&=\frac{l+m+n+1-\sqrt{l^2+1}-\sqrt{m^2+1}-\sqrt{n^2+1}}{2l}
,\\
y&=\frac{l+m+n+1+\sqrt{l^2+1}+\sqrt{m^2+1}-\sqrt{n^2+1}}{2m}
,\\
z&=\frac{l+m+n+1+\sqrt{l^2+1}-\sqrt{m^2+1}+\sqrt{n^2+1}}{2n}
.
\end{aligned}
\right.
\label{lmn:i6}
\displaybreak[0]\\
\left\{
\begin{aligned}
u&=\frac{\sqrt{l^2+1}-l-1}{2}
,\\
v&=-\frac{\sqrt{m^2+1}+m+1}{2}
,\\
w&=\frac{\sqrt{n^2+1}-n-1}{2}
,\\
x&=\frac{l+m+n+1+\sqrt{l^2+1}+\sqrt{m^2+1}-\sqrt{n^2+1}}{2l}
,\\
y&=\frac{l+m+n+1-\sqrt{l^2+1}-\sqrt{m^2+1}-\sqrt{n^2+1}}{2m}
,\\
z&=\frac{l+m+n+1-\sqrt{l^2+1}+\sqrt{m^2+1}+\sqrt{n^2+1}}{2n}
.
\end{aligned}
\right.
\label{lmn:i7}
\displaybreak[0]\\
\left\{
\begin{aligned}
u&=\frac{\sqrt{l^2+1}-l-1}{2}
,\\
v&=\frac{\sqrt{m^2+1}-m-1}{2}
,\\
w&=-\frac{\sqrt{n^2+1}+n+1}{2}
,\\
x&=\frac{l+m+n+1+\sqrt{l^2+1}-\sqrt{m^2+1}+\sqrt{n^2+1}}{2l}
,\\
y&=\frac{l+m+n+1-\sqrt{l^2+1}+\sqrt{m^2+1}+\sqrt{n^2+1}}{2m}
,\\
z&=\frac{l+m+n+1-\sqrt{l^2+1}-\sqrt{m^2+1}-\sqrt{n^2+1}}{2n}
.
\end{aligned}
\right.
\label{lmn:i8}
\end{gather}

Define $\alpha,\beta,\gamma\in(0,\pi/2)$ and $\sigma$ by
\begin{align*}
\sin^2\alpha &= \frac{a}{s}
,&
\sin^2\beta &= \frac{b}{s}
,&
\sin^2\gamma &= \frac{c}{s}
,&
\sigma &= \frac{\alpha+\beta+\gamma}{2}
.
\end{align*}

The fourth equation in \eqref{lmn:i1} corresponds to a value of $r_1$ as follows.
\[
r_1=\frac{r\bigl(l+m+n-1+\sqrt{l^2+1}-\sqrt{m^2+1}-\sqrt{n^2+1}\bigr)}{2l}.
\]
Since
\begin{gather*}
r = \sqrt{\frac{(s-a)(s-b)(s-c)}{s}}
,\\
l = \frac{s-a}{r} = \frac{s}{\rA}
,\qquad
m = \frac{s-b}{r} = \frac{s}{\rB}
,\qquad
n = \frac{s-c}{r} = \frac{s}{\rC}
,
\end{gather*}
it holds that
\begin{align*}
&
\frac{r\bigl(l+m+n-1+\sqrt{l^2+1}-\sqrt{m^2+1}-\sqrt{n^2+1}\bigr)}{2l}
\\
&=
\frac{\rA}{2}
\mathopen{}\left(
1
-\sqrt{\frac{(s-a)(s-b)(s-c)}{s^3}}
+\sqrt{\frac{(s-a)bc}{s^3}}
-\sqrt{\frac{a(s-b)c}{s^3}}
-\sqrt{\frac{ab(s-c)}{s^3}}
\right)\mathclose{}
\\
&=
\frac{\rA}{2}
(
1
-\cos\alpha\cos\beta\cos\gamma
+\cos\alpha\sin\beta\sin\gamma
-\sin\alpha\cos\beta\sin\gamma
-\sin\alpha\sin\beta\cos\gamma
)
\\
&=
\frac{\rA(1-\cos(\beta+\gamma-\alpha))}{2}
\\
&=
\rA\sin^2(\sigma-\alpha)
.
\end{align*}

By making similar calculations on
every last three equations in
\eqref{lmn:i1},
\eqref{lmn:i2},
\eqref{lmn:i3},
\eqref{lmn:i4},
\eqref{lmn:i5},
\eqref{lmn:i6},
\eqref{lmn:i7} and
\eqref{lmn:i8},
we obtain the following respective solutions in Case~1.
\begin{gather}
\left\{
\begin{aligned}
r_1 &=
\rA\sin^2(\sigma-\alpha)
,\\
r_2 &=
\rB\sin^2(\sigma-\beta)
,\\
r_3 &=
\rC\sin^2(\sigma-\gamma)
.
\end{aligned}
\right.
\label{schellbach:i1}
\displaybreak[0]\\
\left\{
\begin{aligned}
r_1 &=
\rA\sin^2\sigma
,\\
r_2 &=
\rB\sin^2(\sigma-\gamma)
,\\
r_3 &=
\rC\sin^2(\sigma-\beta)
.
\end{aligned}
\right.
\label{schellbach:i2}
\displaybreak[0]\\
\left\{
\begin{aligned}
r_1 &=
\rA\sin^2(\sigma-\gamma)
,\\
r_2 &=
\rB\sin^2\sigma
,\\
r_3 &=
\rC\sin^2(\sigma-\alpha)
.
\end{aligned}
\right.
\label{schellbach:i3}
\displaybreak[0]\\
\left\{
\begin{aligned}
r_1 &=
\rA\sin^2(\sigma-\beta)
,\\
r_2 &=
\rB\sin^2(\sigma-\alpha)
,\\
r_3 &=
\rC\sin^2\sigma
.
\end{aligned}
\right.
\label{schellbach:i4}
\displaybreak[0]\\
\left\{
\begin{aligned}
r_1 &=
\rA\cos^2(\sigma-\alpha)
,\\
r_2 &=
\rB\cos^2(\sigma-\beta)
,\\
r_3 &=
\rC\cos^2(\sigma-\gamma)
.
\end{aligned}
\right.
\label{schellbach:i5}
\displaybreak[0]\\
\left\{
\begin{aligned}
r_1 &=
\rA\cos^2\sigma
,\\
r_2 &=
\rB\cos^2(\sigma-\gamma)
,\\
r_3 &=
\rC\cos^2(\sigma-\beta)
.
\end{aligned}
\right.
\label{schellbach:i6}
\displaybreak[0]\\
\left\{
\begin{aligned}
r_1 &=
\rA\cos^2(\sigma-\gamma)
,\\
r_2 &=
\rB\cos^2\sigma
,\\
r_3 &=
\rC\cos^2(\sigma-\alpha)
.
\end{aligned}
\right.
\label{schellbach:i7}
\displaybreak[0]\\
\left\{
\begin{aligned}
r_1 &=
\rA\cos^2(\sigma-\beta)
,\\
r_2 &=
\rB\cos^2(\sigma-\alpha)
,\\
r_3 &=
\rC\cos^2\sigma
.
\end{aligned}
\right.
\label{schellbach:i8}
\end{gather}

\subsection{Cases 2 \& 3}
\label{sec:case_23}

In Case~2, we have that the following three disjunctions of equations hold.
\begin{gather*}
\ptB\ptD_2+\ptD_3\ptC+\ptD_2\ptD_3 = \ptB\ptDA + \ptDA\ptC
\quad\text{or}\quad
\ptB\ptD_2+\ptD_3\ptC-\ptD_2\ptD_3 = \ptB\ptDA + \ptDA\ptC
,
\\
\ptA\ptE_1-\ptC\ptE_3+\ptE_1\ptE_3 = \ptA\ptEA - \ptC\ptEA
\quad\text{or}\quad
\ptA\ptE_1-\ptC\ptE_3-\ptE_1\ptE_3 = \ptA\ptEA - \ptC\ptEA
,
\\
\ptA\ptF_1-\ptB\ptF_2+\ptF_1\ptF_2 = \ptA\ptFA - \ptB\ptFA
\quad\text{or}\quad
\ptA\ptF_1-\ptB\ptF_2-\ptF_1\ptF_2 = \ptA\ptFA - \ptB\ptFA
.
\end{gather*}
By expressing the lengths by the radii and the angle sizes,
we obtain from the first disjunction that
\begin{gather}
r_2\tan\frac{B}{2} + r_3\tan\frac{C}{2} + 2\sqrt{r_2r_3} = \rA\tan\frac{B}{2}+\rA\tan\frac{C}{2}
\label{eq:exA.BC+}
\shortintertext{or}
r_2\tan\frac{B}{2} + r_3\tan\frac{C}{2} - 2\sqrt{r_2r_3} = \
rA\tan\frac{B}{2}+\rA\tan\frac{C}{2}
\label{eq:exA.BC-}
,
\end{gather}
we obtain from the second disjunction that
\begin{gather}
r_1\cot\frac{A}{2} - r_3\tan\frac{C}{2} + 2\sqrt{r_1r_3} = \rA\cot\frac{A}{2}-\rA\tan\frac{C}{2}
\label{eq:exA.AC+}
\shortintertext{or}
r_1\cot\frac{A}{2} - r_3\tan\frac{C}{2} - 2\sqrt{r_1r_3} = \rA\cot\frac{A}{2}-\rA\tan\frac{C}{2}
\label{eq:exA.AC-}
,
\end{gather}
and we obtain from the third disjunction that
\begin{gather}
r_1\cot\frac{A}{2} - r_2\tan\frac{B}{2} + 2\sqrt{r_1r_2} = \rA\cot\frac{A}{2}-\rA\tan\frac{B}{2}
\label{eq:exA.AB+}
\shortintertext{or}
r_1\cot\frac{A}{2} - r_2\tan\frac{B}{2} - 2\sqrt{r_1r_2} = \rA\cot\frac{A}{2}-\rA\tan\frac{B}{2}
\label{eq:exA.AB-}
.
\end{gather}

Define $l$, $\bar{m}$, $\bar{n}$ by
\begin{align}
l &= \cot\frac{A}{2}
,&
\bar{m} &= \tan\frac{B}{2}
,&
\bar{n} &= \tan\frac{C}{2}
.
\label{eq:exA.lmnABC}
\end{align}
Define $u$, $v$, $w$, $x$, $y$, $z$ by
\begin{gather*}
u =
\begin{dcases}
-\frac{\sqrt{r_2r_3}}{\rA} & \text{if \eqref{eq:exA.BC+} holds,}
\\
 \frac{\sqrt{r_2r_3}}{\rA} & \text{if \eqref{eq:exA.BC-} holds,}
\end{dcases}
\displaybreak[0]\\
v =
\begin{dcases}
 \frac{\sqrt{r_1r_3}}{\rA} & \text{if \eqref{eq:exA.AC+} holds,}
\\
-\frac{\sqrt{r_1r_3}}{\rA} & \text{if \eqref{eq:exA.AC-} holds,}
\end{dcases}
\displaybreak[0]\\
w =
\begin{dcases}
 \frac{\sqrt{r_1r_2}}{\rA} & \text{if \eqref{eq:exA.AB+} holds,}
\\
-\frac{\sqrt{r_1r_2}}{\rA} & \text{if \eqref{eq:exA.AB-} holds,}
\end{dcases}
\displaybreak[0]\\
x = \frac{r_1}{\rA},
\qquad
y = \frac{r_2}{\rA},
\qquad
z = \frac{r_3}{\rA}.
\end{gather*}
Then we have
\begin{equation}
\left\{
\begin{aligned}
\bar{m}y+\bar{n}z-2u &= \bar{m}+\bar{n},
\\
lx-\bar{n}z+2v &= l-\bar{n},
\\
lx-\bar{m}y+2w &= l-\bar{m},
\\
xy &= w^2,
\\
xz &= v^2,
\\
yz &= u^2.
\end{aligned}
\right.
\label{eq:exA}
\end{equation}

In Case~3, we have
\begin{gather*}
-\ptB\ptD_2-\ptD_3\ptC+\ptD_2\ptD_3 = \ptB\ptDA + \ptDA\ptC
,
\\
-\ptA\ptE_1+\ptC\ptE_3+\ptE_1\ptE_3 = \ptA\ptEA - \ptC\ptEA
\;\text{or}\;
-\ptA\ptE_1+\ptC\ptE_3-\ptE_1\ptE_3 = \ptA\ptEA - \ptC\ptEA
,
\\
-\ptA\ptF_1+\ptB\ptF_2+\ptF_1\ptF_2 = \ptA\ptFA - \ptB\ptFA
\;\text{or}\;
-\ptA\ptF_1+\ptB\ptF_2-\ptF_1\ptF_2 = \ptA\ptFA - \ptB\ptFA
.
\end{gather*}
By expressing the lengths by the radii and the angle sizes,
we obtain from the first equation that
\begin{equation}
-r_2\tan\frac{B}{2} - r_3\tan\frac{C}{2} + 2\sqrt{r_2r_3} = \rA\tan\frac{B}{2}+\rA\tan\frac{C}{2}
\label{eq:exA*.BC+}
,
\end{equation}
we obtain form the second conjunction that
\begin{gather}
-r_1\cot\frac{A}{2} + r_3\tan\frac{C}{2} + 2\sqrt{r_1r_3} = \rA\cot\frac{A}{2}-\rA\tan\frac{C}{2}
\label{eq:exA*.AC+}
\shortintertext{or}
-r_1\cot\frac{A}{2} + r_3\tan\frac{C}{2} - 2\sqrt{r_1r_3} = \rA\cot\frac{A}{2}-\rA\tan\frac{C}{2}
\label{eq:exA*.AC-}
,
\end{gather}
and we obtain from the third conjunction that
\begin{gather}
-r_1\cot\frac{A}{2} + r_2\tan\frac{B}{2} + 2\sqrt{r_1r_2} = \rA\cot\frac{A}{2}-\rA\tan\frac{B}{2}
\label{eq:exA*.AB+}
\shortintertext{or}
-r_1\cot\frac{A}{2} + r_2\tan\frac{B}{2} - 2\sqrt{r_1r_2} = \rA\cot\frac{A}{2}-\rA\tan\frac{B}{2}
\label{eq:exA*.AB-}
.
\end{gather}

Define $l$, $\bar{m}$, $\bar{n}$ by \eqref{eq:exA.lmnABC}.
Define $u$, $v$, $w$, $x$, $y$, $z$ by
\begin{gather*}
u =
-\frac{\sqrt{r_2r_3}}{\rA}
\displaybreak[0]\\
v =
\begin{dcases}
 \frac{\sqrt{r_1r_3}}{\rA} & \text{if \eqref{eq:exA*.AC+} holds,}
\\
-\frac{\sqrt{r_1r_3}}{\rA} & \text{if \eqref{eq:exA*.AC-} holds,}
\end{dcases}
\displaybreak[0]\\
w =
\begin{dcases}
 \frac{\sqrt{r_1r_2}}{\rA} & \text{if \eqref{eq:exA*.AB+} holds,}
\\
-\frac{\sqrt{r_1r_2}}{\rA} & \text{if \eqref{eq:exA*.AB-} holds,}
\end{dcases}
\displaybreak[0]\\
x = -\frac{r_1}{\rA},
\qquad
y = -\frac{r_2}{\rA},
\qquad
z = -\frac{r_3}{\rA}.
\end{gather*}
Then we have the same system of equations as \eqref{eq:exA}.

For any triangle $\ptA\ptB\ptC$,
if $l$, $\bar{m}$, $\bar{n}$ are defined by \eqref{eq:exA.lmnABC},
then $l\bar{m}\bar{n}={l-\bar{m}-\bar{n}}$ holds.
On the other hand,
if positive reals $l$, $\bar{m}$, $\bar{n}$ satisfy $l\bar{m}\bar{n}={l-\bar{m}-\bar{n}}$,
then there exists a triangle $\ptA\ptB\ptC$
that satisfies \eqref{eq:exA.lmnABC}.
Thus,
Case~2 and Case~3 can be unified and reduced into solving the system of equations \eqref{eq:exA}
for $u$, $v$, $w$, $x$, $y$, $z$
with positive real parameters $l$, $\bar{m}$, $\bar{n}$
under the restriction $l\bar{m}\bar{n}={l-\bar{m}-\bar{n}}$.

As we will show in Appendix~\ref{sec:how-to-solve-2}.
the system of equations has the following $8$ solutions.
\begin{gather}
\left\{
\begin{aligned}
u&=\frac{\sqrt{l^2+1}-l-1}{2}
,\\
v&=\frac{\sqrt{\bar{m}^2+1}+\bar{m}-1}{2}
,\\
w&=\frac{\sqrt{\bar{n}^2+1}+\bar{n}-1}{2}
,\\
x&=\frac{\sqrt{l^2+1}-\sqrt{\bar{m}^2+1}-\sqrt{\bar{n}^2+1}+l-\bar{m}-\bar{n}+1}{2l}
,\\
y&=\frac{\sqrt{l^2+1}-\sqrt{\bar{m}^2+1}+\sqrt{\bar{n}^2+1}-l+\bar{m}+\bar{n}-1}{2\bar{m}}
,\\
z&=\frac{\sqrt{l^2+1}+\sqrt{\bar{m}^2+1}-\sqrt{\bar{n}^2+1}-l+\bar{m}+\bar{n}-1}{2\bar{n}}
.
\end{aligned}
\right.
\label{lmn:a1}
\displaybreak[0]\\
\left\{
\begin{aligned}
u&=\frac{\sqrt{l^2+1}-l-1}{2}
,\\
v&=-\frac{\sqrt{\bar{m}^2+1}-\bar{m}+1}{2}
,\\
w&=-\frac{\sqrt{\bar{n}^2+1}-\bar{n}+1}{2}
,\\
x&=\frac{\sqrt{l^2+1}+\sqrt{\bar{m}^2+1}+\sqrt{\bar{n}^2+1}+l-\bar{m}-\bar{n}+1}{2l}
,\\
y&=\frac{\sqrt{l^2+1}+\sqrt{\bar{m}^2+1}-\sqrt{\bar{n}^2+1}-l+\bar{m}+\bar{n}-1}{2\bar{m}}
,\\
z&=\frac{\sqrt{l^2+1}-\sqrt{\bar{m}^2+1}+\sqrt{\bar{n}^2+1}-l+\bar{m}+\bar{n}-1}{2\bar{n}}
.
\end{aligned}
\right.
\label{lmn:a2}
\displaybreak[0]\\
\left\{
\begin{aligned}
u&=\frac{\sqrt{l^2+1}-l+1}{2}
,\\
v&=-\frac{\sqrt{\bar{m}^2+1}-\bar{m}-1}{2}
,\\
w&=\frac{\sqrt{\bar{n}^2+1}+\bar{n}+1}{2}
,\\
x&=\frac{\sqrt{l^2+1}+\sqrt{\bar{m}^2+1}-\sqrt{\bar{n}^2+1}+l-\bar{m}-\bar{n}-1}{2l}
,\\
y&=\frac{\sqrt{l^2+1}+\sqrt{\bar{m}^2+1}+\sqrt{\bar{n}^2+1}-l+\bar{m}+\bar{n}+1}{2\bar{m}}
,\\
z&=\frac{\sqrt{l^2+1}-\sqrt{\bar{m}^2+1}-\sqrt{\bar{n}^2+1}-l+\bar{m}+\bar{n}+1}{2\bar{n}}
.
\end{aligned}
\right.
\label{lmn:a3}
\displaybreak[0]\\
\left\{
\begin{aligned}
u&=\frac{\sqrt{l^2+1}-l+1}{2}
,\\
v&=\frac{\sqrt{\bar{m}^2+1}+\bar{m}+1}{2}
,\\
w&=-\frac{\sqrt{\bar{n}^2+1}-\bar{n}-1}{2}
,\\
x&=\frac{\sqrt{l^2+1}-\sqrt{\bar{m}^2+1}+\sqrt{\bar{n}^2+1}+l-\bar{m}-\bar{n}-1}{2l}
,\\
y&=\frac{\sqrt{l^2+1}-\sqrt{\bar{m}^2+1}-\sqrt{\bar{n}^2+1}-l+\bar{m}+\bar{n}+1}{2\bar{m}}
,\\
z&=\frac{\sqrt{l^2+1}+\sqrt{\bar{m}^2+1}+\sqrt{\bar{n}^2+1}-l+\bar{m}+\bar{n}+1}{2\bar{n}}
.
\end{aligned}
\right.
\label{lmn:a4}
\displaybreak[0]\\
\left\{
\begin{aligned}
u&=-\frac{\sqrt{l^2+1}+l-1}{2}
,\\
v&=-\frac{\sqrt{\bar{m}^2+1}-\bar{m}-1}{2}
,\\
w&=-\frac{\sqrt{\bar{n}^2+1}-\bar{n}-1}{2}
,\\
x&=-\frac{\sqrt{l^2+1}-\sqrt{\bar{m}^2+1}-\sqrt{\bar{n}^2+1}-l+\bar{m}+\bar{n}+1}{2l}
,\\
y&=-\frac{\sqrt{l^2+1}-\sqrt{\bar{m}^2+1}+\sqrt{\bar{n}^2+1}+l-\bar{m}-\bar{n}-1}{2\bar{m}}
,\\
z&=-\frac{\sqrt{l^2+1}+\sqrt{\bar{m}^2+1}-\sqrt{\bar{n}^2+1}+l-\bar{m}-\bar{n}-1}{2\bar{n}}
.
\end{aligned}
\right.
\label{lmn:a5}
\displaybreak[0]\\
\left\{
\begin{aligned}
u&=-\frac{\sqrt{l^2+1}+l-1}{2}
,\\
v&=\frac{\sqrt{\bar{m}^2+1}+\bar{m}+1}{2}
,\\
w&=\frac{\sqrt{\bar{n}^2+1}+\bar{n}+1}{2}
,\\
x&=-\frac{\sqrt{l^2+1}+\sqrt{\bar{m}^2+1}+\sqrt{\bar{n}^2+1}-l+\bar{m}+\bar{n}+1}{2l}
,\\
y&=-\frac{\sqrt{l^2+1}+\sqrt{\bar{m}^2+1}-\sqrt{\bar{n}^2+1}+l-\bar{m}-\bar{n}-1}{2\bar{m}}
,\\
z&=-\frac{\sqrt{l^2+1}-\sqrt{\bar{m}^2+1}+\sqrt{\bar{n}^2+1}+l-\bar{m}-\bar{n}-1}{2\bar{n}}
.
\end{aligned}
\right.
\label{lmn:a6}
\displaybreak[0]\\
\left\{
\begin{aligned}
u&=-\frac{\sqrt{l^2+1}+l+1}{2}
,\\
v&=\frac{\sqrt{\bar{m}^2+1}+\bar{m}-1}{2}
,\\
w&=-\frac{\sqrt{\bar{n}^2+1}-\bar{n}+1}{2}
,\\
x&=-\frac{\sqrt{l^2+1}+\sqrt{\bar{m}^2+1}-\sqrt{\bar{n}^2+1}-l+\bar{m}+\bar{n}-1}{2l}
,\\
y&=-\frac{\sqrt{l^2+1}+\sqrt{\bar{m}^2+1}+\sqrt{\bar{n}^2+1}+l-\bar{m}-\bar{n}+1}{2\bar{m}}
,\\
z&=-\frac{\sqrt{l^2+1}-\sqrt{\bar{m}^2+1}-\sqrt{\bar{n}^2+1}+l-\bar{m}-\bar{n}+1}{2\bar{n}}
.
\end{aligned}
\right.
\label{lmn:a7}
\displaybreak[0]\\
\left\{
\begin{aligned}
u&=-\frac{\sqrt{l^2+1}+l+1}{2}
,\\
v&=-\frac{\sqrt{\bar{m}^2+1}-\bar{m}+1}{2}
,\\
w&=\frac{\sqrt{\bar{n}^2+1}+\bar{n}-1}{2}
,\\
x&=-\frac{\sqrt{l^2+1}-\sqrt{\bar{m}^2+1}+\sqrt{\bar{n}^2+1}-l+\bar{m}+\bar{n}-1}{2l}
,\\
y&=-\frac{\sqrt{l^2+1}-\sqrt{\bar{m}^2+1}-\sqrt{\bar{n}^2+1}+l-\bar{m}-\bar{n}+1}{2\bar{m}}
,\\
z&=-\frac{\sqrt{l^2+1}+\sqrt{\bar{m}^2+1}+\sqrt{\bar{n}^2+1}+l-\bar{m}-\bar{n}+1}{2\bar{n}}
.
\end{aligned}
\right.
\label{lmn:a8}
\end{gather}

Define $\alphaA,\betaA,\gammaA\in(0,+\infty)$ and $\sigmaA$ by
\begin{gather*}
\sinh^2\alphaA = \frac{a}{s-a}
,\qquad
\sinh^2\betaA = \frac{s-c}{s-a}
,\qquad
\sinh^2\gammaA = \frac{s-b}{s-a}
,\\
\sigmaA = \frac{\alphaA+\betaA+\gammaA}{2}
.
\end{gather*}

The fourth equation in \eqref{lmn:a1} corresponds to a value of $r_1$ as follows.
\[
r_1=\frac{\rA\bigl(\sqrt{l^2+1}-\sqrt{\bar{m}^2+1}-\sqrt{\bar{n}^2+1}+l-\bar{m}-\bar{n}+1\bigr)}{2l}
\]
Since
\begin{gather*}
\rA = \sqrt{\frac{s(s-b)(s-c)}{s-a}}
,\\
l = \frac{s}{\rA} = \frac{s-a}{r}
,\qquad
\bar{m} = \frac{s-c}{\rA} = \frac{s-a}{\rC}
,\qquad
\bar{n} = \frac{s-b}{\rA} = \frac{s-a}{\rB}
,
\end{gather*}
it holds that
\begin{align*}
&
\frac{\rA\bigl(\sqrt{l^2+1}-\sqrt{\bar{m}^2+1}-\sqrt{\bar{n}^2+1}+l-\bar{m}-\bar{n}+1\bigr)}{2l}
\\
&=
\frac{r}{2}
\Biggl(
\sqrt{\frac{bcs}{(s-a)^3}}
-\sqrt{\frac{ab(s-b)}{(s-a)^3}}
-\sqrt{\frac{ac(s-c)}{(s-a)^3}}
+\sqrt{\frac{s(s-b)(s-c)}{(s-a)^3}}
+1
\Biggr)
\\
&=
\frac{r}{2}
(
\cosh\alphaA\cosh\betaA\cosh\gammaA
-\sinh\alphaA\cosh\betaA\sinh\gammaA
\\
&\phantom{=\frac{r}{2}(}\;
-\sinh\alphaA\sinh\betaA\cosh\gammaA
+\cosh\alphaA\sinh\betaA\sinh\gammaA
+1
)
\\
&=
\frac{\rA(\cosh(\betaA+\gammaA-\alphaA)+1)}{2}
\\
&=
\rA\sinh^2(\sigmaA-\alphaA)
.
\end{align*}

By making similar calculations on
every last three equations in
\eqref{lmn:a1},
\eqref{lmn:a2},
\eqref{lmn:a3},
\eqref{lmn:a4},
\eqref{lmn:a5},
\eqref{lmn:a6},
\eqref{lmn:a7} and
\eqref{lmn:a8},
we obtain the following respective solutions in Cases~2 and~3.
\begin{gather}
\left\{
\begin{aligned}
r_1 &=
r\cosh^2(\sigmaA-\alphaA)
,\\
r_2 &=
\rC\sinh^2(\sigmaA-\betaA)
,\\
r_3 &=
\rB\sinh^2(\sigmaA-\gammaA)
.
\end{aligned}
\right.
\label{schellbach:a1}
\displaybreak[0]\\
\left\{
\begin{aligned}
r_1 &=
r\cosh^2\sigmaA
,\\
r_2 &=
\rC\sinh^2(\sigmaA-\gammaA)
,\\
r_3 &=
\rB\sinh^2(\sigmaA-\betaA)
.
\end{aligned}
\right.
\label{schellbach:a2}
\displaybreak[0]\\
\left\{
\begin{aligned}
r_1 &=
r\cosh^2(\sigmaA-\gammaA)
,\\
r_2 &=
\rC\sinh^2\sigmaA
,\\
r_3 &=
\rB\sinh^2(\sigmaA-\alphaA)
.
\end{aligned}
\right.
\label{schellbach:a3}
\displaybreak[0]\\
\left\{
\begin{aligned}
r_1 &=
r\cosh^2(\sigmaA-\betaA)
,\\
r_2 &=
\rC\sinh^2(\sigmaA-\alphaA)
,\\
r_3 &=
\rB\sinh^2\sigmaA
.
\end{aligned}
\right.
\label{schellbach:a4}
\displaybreak[0]\\
\left\{
\begin{aligned}
r_1 &=
r\sinh^2(\sigmaA-\alphaA)
,\\
r_2 &=
\rC\cosh^2(\sigmaA-\betaA)
,\\
r_3 &=
\rB\cosh^2(\sigmaA-\gammaA)
.
\end{aligned}
\right.
\label{schellbach:a5}
\displaybreak[0]\\
\left\{
\begin{aligned}
r_1 &=
r\sinh^2\sigmaA
,\\
r_2 &=
\rC\cosh^2(\sigmaA-\gammaA)
,\\
r_3 &=
\rB\cosh^2(\sigmaA-\betaA)
.
\end{aligned}
\right.
\label{schellbach:a6}
\displaybreak[0]\\
\left\{
\begin{aligned}
r_1 &=
r\sinh^2(\sigmaA-\gammaA)
,\\
r_2 &=
\rC\cosh^2\sigmaA
,\\
r_3 &=
\rB\cosh^2(\sigmaA-\alphaA)
.
\end{aligned}
\right.
\label{schellbach:a7}
\displaybreak[0]\\
\left\{
\begin{aligned}
r_1 &=
r\sinh^2(\sigmaA-\betaA)
,\\
r_2 &=
\rC\cosh^2(\sigmaA-\alphaA)
,\\
r_3 &=
\rB\cosh^2\sigmaA
.
\end{aligned}
\right.
\label{schellbach:a8}
\end{gather}

\subsection{Cases 4 \& 5}
Define $\alphaB,\betaB,\gammaB\in(0,+\infty)$ and $\sigmaB$ by
\begin{gather*}
\sinh^2\alphaB = \frac{s-c}{s-b}
,\qquad
\sinh^2\betaB = \frac{b}{s-b}
,\qquad
\sinh^2\gammaB = \frac{s-a}{s-b}
,\\
\sigmaB = \frac{\alphaB+\betaB+\gammaB}{2}
.
\end{gather*}

Analogously to \ref{sec:case_23},
we obtain the following solutions in Cases~4 and~5.
\begin{gather}
\left\{
\begin{aligned}
r_1 &=
\rC\sinh^2(\sigmaB-\alphaB)
,\\
r_2 &=
r\cosh^2(\sigmaB-\betaB)
,\\
r_3 &=
\rA\sinh^2(\sigmaB-\gammaB)
.
\end{aligned}
\right.
\label{schellbach:b1}
\displaybreak[0]\\
\left\{
\begin{aligned}
r_1 &=
\rC\sinh^2\sigmaB
,\\
r_2 &=
r\cosh^2(\sigmaB-\gammaB)
,\\
r_3 &=
\rA\sinh^2(\sigmaB-\betaB)
.
\end{aligned}
\right.
\label{schellbach:b2}
\displaybreak[0]\\
\left\{
\begin{aligned}
r_1 &=
\rC\sinh^2(\sigmaB-\gammaB)
,\\
r_2 &=
r\cosh^2\sigmaB
,\\
r_3 &=
\rA\sinh^2(\sigmaB-\alphaB)
.
\end{aligned}
\right.
\label{schellbach:b3}
\displaybreak[0]\\
\left\{
\begin{aligned}
r_1 &=
\rC\sinh^2(\sigmaB-\betaB)
,\\
r_2 &=
r\cosh^2(\sigmaB-\alphaB)
,\\
r_3 &=
\rA\sinh^2\sigmaB
.
\end{aligned}
\right.
\label{schellbach:b4}
\displaybreak[0]\\
\left\{
\begin{aligned}
r_1 &=
\rC\cosh^2(\sigmaB-\alphaB)
,\\
r_2 &=
r\sinh^2(\sigmaB-\betaB)
,\\
r_3 &=
\rA\cosh^2(\sigmaB-\gammaB)
.
\end{aligned}
\right.
\label{schellbach:b5}
\displaybreak[0]\\
\left\{
\begin{aligned}
r_1 &=
\rC\cosh^2\sigmaB
,\\
r_2 &=
r\sinh^2(\sigmaB-\gammaB)
,\\
r_3 &=
\rB\cosh^2(\sigmaB-\betaB)
.
\end{aligned}
\right.
\label{schellbach:b6}
\displaybreak[0]\\
\left\{
\begin{aligned}
r_1 &=
\rC\cosh^2(\sigmaB-\gammaB)
,\\
r_2 &=
r\sinh^2\sigmaB
,\\
r_3 &=
\rA\cosh^2(\sigmaB-\alphaB)
.
\end{aligned}
\right.
\label{schellbach:b7}
\displaybreak[0]\\
\left\{
\begin{aligned}
r_1 &=
\rC\cosh^2(\sigmaB-\betaB)
,\\
r_2 &=
r\sinh^2(\sigmaB-\alphaB)
,\\
r_3 &=
\rA\cosh^2\sigmaB
.
\end{aligned}
\right.
\label{schellbach:b8}
\end{gather}

\subsection{Cases 6 \& 7}
Define $\alphaC,\betaC,\gammaC\in(0,+\infty)$ and $\sigmaC$ by
\begin{gather*}
\sinh^2\alphaC = \frac{s-b}{s-c}
,\qquad
\sinh^2\betaC = \frac{s-a}{s-c}
,\qquad
\sinh^2\gammaC = \frac{c}{s-c}
,\\
\sigmaC = \frac{\alphaC+\betaC+\gammaC}{2}
.
\end{gather*}

Analogously to \ref{sec:case_23},
we obtain the following solutions in Cases~6 and~7.
\begin{gather}
\left\{
\begin{aligned}
r_1 &=
\rB\sinh^2(\sigmaC-\alphaC)
,\\
r_2 &=
\rA\sinh^2(\sigmaC-\betaC)
,\\
r_3 &=
r\cosh^2(\sigmaC-\gammaC)
.
\end{aligned}
\right.
\label{schellbach:c1}
\displaybreak[0]\\
\left\{
\begin{aligned}
r_1 &=
\rB\sinh^2\sigmaC
,\\
r_2 &=
\rA\sinh^2(\sigmaC-\gammaC)
,\\
r_3 &=
r\cosh^2(\sigmaC-\betaC)
.
\end{aligned}
\right.
\label{schellbach:c2}
\displaybreak[0]\\
\left\{
\begin{aligned}
r_1 &=
\rB\sinh^2(\sigmaC-\gammaC)
,\\
r_2 &=
\rA\sinh^2\sigmaC
,\\
r_3 &=
r\cosh^2(\sigmaC-\alphaC)
.
\end{aligned}
\right.
\label{schellbach:c3}
\displaybreak[0]\\
\left\{
\begin{aligned}
r_1 &=
\rB\sinh^2(\sigmaC-\betaC)
,\\
r_2 &=
\rA\sinh^2(\sigmaC-\alphaC)
,\\
r_3 &=
r\cosh^2\sigmaC
.
\end{aligned}
\right.
\label{schellbach:c4}
\displaybreak[0]\\
\left\{
\begin{aligned}
r_1 &=
\rB\cosh^2(\sigmaC-\alphaC)
,\\
r_2 &=
\rA\cosh^2(\sigmaC-\betaC)
,\\
r_3 &=
r\sinh^2(\sigmaC-\gammaC)
.
\end{aligned}
\right.
\label{schellbach:c5}
\displaybreak[0]\\
\left\{
\begin{aligned}
r_1 &=
\rB\cosh^2\sigmaC
,\\
r_2 &=
\rA\cosh^2(\sigmaC-\gammaC)
,\\
r_3 &=
r\sinh^2(\sigmaC-\betaC)
.
\end{aligned}
\right.
\label{schellbach:c6}
\displaybreak[0]\\
\left\{
\begin{aligned}
r_1 &=
\rB\cosh^2(\sigmaC-\gammaC)
,\\
r_2 &=
\rA\cosh^2\sigmaC
,\\
r_3 &=
r\sinh^2(\sigmaC-\betaC)
.
\end{aligned}
\right.
\label{schellbach:c7}
\displaybreak[0]\\
\left\{
\begin{aligned}
r_1 &=
\rB\cosh^2(\sigmaC-\betaC)
,\\
r_2 &=
\rA\cosh^2(\sigmaC-\alphaC)
,\\
r_3 &=
r\sinh^2\sigmaC
.
\end{aligned}
\right.
\label{schellbach:c8}
\end{gather}

\section{Conclusion}
\begin{theorem}
For any triangle,
there exist 32 triplets of circles
such that each circle is tangent to
the other two circles
and to two of the sides of the reference triangle or their extensions.
The radii can be expressed by
\eqref{schellbach:i1}--\eqref{schellbach:i8},
\eqref{schellbach:a1}--\eqref{schellbach:a8},
\eqref{schellbach:b1}--\eqref{schellbach:b8},
\eqref{schellbach:c1}--\eqref{schellbach:c8}.
\end{theorem}

Figures~\ref{fig:i1}--\ref{fig:c8} illustrate
the 32 triplets of circles
for a triangle $\ptA\ptB\ptC$ such that
$A=45^\circ$,
$B=54^\circ$,
$C=81^\circ$.

\appendix
\section{Solutions of the systems of equations}
\label{sec:how-to-solve-1}
\label{sec:how-to-solve-2}
In this appendix,
we will solve some systems of equations by computing Gr\"obner bases.
Although it is difficult to compute the Gr\"obner bases by hand,
any computer algebra system that can compute Gr\"obner bases should work.

\begin{proposition}
The system of equations~\eqref{eq:in}
for the variables $u$, $v$, $w$, $x$, $y$, $z$
with the positive real parameters $l$, $m$, $n$
under the restriction $lmn={l+m+n}$.
has $8$ solutions
\eqref{lmn:i1}--\eqref{lmn:i8}.
\end{proposition}
\begin{proof}
Counting $l$, $m$, $n$ among the variables
in addition to $u$, $v$, $w$, $x$, $y$, $z$,
we compute the reduced Gr\"obner basis of
$\{
my+nz+2u-m-n,\linebreak[0]\;
lx+nz+2v-l-n,\linebreak[0]\;
lx+my+2w-l-m,\linebreak[0]\;
xy-w^2,\linebreak[0]\;
xz-v^2,\linebreak[0]\;
yz-u^2,\linebreak[0]\;
lmn-l-m-n
\}$
with the degree reverse lexicographical ordering
$x \succ y \succ z \succ u \succ v \succ w \succ l \succ m \succ n$.
The reduced Gr\"obner basis consists of $67$ polynomials
including
\begin{align*}
f_1 &= (2u^2+2lu-2u-l)(2u^2+2lu+2u+l)
,\\
f_2 &= (2v^2+2mv-2v-m)(2v^2+2mv+2v+m)
,\\
f_3 &= (2w^2+2nw-2w-n)(2w^2+2nw+2w+n)
,\\
f_4 &= lx-u+v+w-l
,\\
f_5 &= my+u-v+w-m
,\\
f_6 &= nz+u+v-w-n
.
\end{align*}

By solving $\{f_1=0,\; f_2=0,\; f_3=0,\; f_4=0,\; f_5=0,\; f_6=0 \}$
for $u$, $v$, $w$, $x$, $y$, $z$,
we obtain $128$ solutions.
Note that $lmn=l+m+n$ is equivalent to
\begin{equation}
l = \frac{m+n}{mn-1}.
\label{sol:in.l.mn}
\end{equation}
By assigning each of the $128$ solutions together with \eqref{sol:in.l.mn} to
$\{
my+nz+2u-m-n,\linebreak[0]\;
lx+nz+2v-l-n,\linebreak[0]\;
lx+my+2w-l-m,\linebreak[0]\;
xy-w^2,\linebreak[0]\;
xz-v^2,\linebreak[0]\;
yz-u^2
\}$
and then picking out the solutions
such that the assignment makes all of the polynomials equal $0$,
we still have $8$ solutions
\eqref{lmn:i1}--\eqref{lmn:i8},
which are the solutions of~\eqref{eq:in}.
\end{proof}

\begin{proposition}
The system of equations~\eqref{eq:exA}
for the variables $u$, $v$, $w$, $x$, $y$, $z$
with the positive real parameters $l$, $\bar{m}$, $\bar{n}$
under the restriction $l\bar{m}\bar{n}={l-\bar{m}-\bar{n}}$
has $8$ solutions
\eqref{lmn:a1}--\eqref{lmn:a8}.
\end{proposition}
\begin{proof}
Counting $l$, $\bar{m}$, $\bar{n}$ among the variables
in addition to $u$, $v$, $w$, $x$, $y$, $z$,
we compute the reduced Gr\"obner basis of
$\{
\bar{m}y+\bar{n}z-2u-\bar{m}-\bar{n},\linebreak[0]\;
lx-\bar{n}z+2v-l+\bar{n},\linebreak[0]\;
lx-\bar{m}y+2w-l+\bar{m},\linebreak[0]\;
xy-w^2,\linebreak[0]\;
xz-v^2,\linebreak[0]\;
yz-u^2,\linebreak[0]\;
l\bar{m}\bar{n}-l+\bar{m}+\bar{n}
\}$
with the degree reverse lexicographical ordering
$x \succ y \succ z \succ u \succ v \succ w \succ l \succ \bar{m} \succ \bar{n}$.
The reduced Gr\"obner basis consists of $67$ polynomials
including
\begin{align*}
f_1 &= (2u^2+2lu-2u-l)(2u^2+2lu+2u+l)
,\\
f_2 &= (2v^2-2\bar{m}v-2v+\bar{m})(2v^2-2\bar{m}v+2v-\bar{m})
,\\
f_3 &= (2w^2-2\bar{n}w-2w+\bar{n})(2w^2-2\bar{n}w+2w-\bar{n})
,\\
f_4 &= lx-u+v+w-l
,\\
f_5 &= \bar{m}y-u+v-w-\bar{m}
,\\
f_6 &= \bar{n}z-u-v+w-\bar{n}
.
\end{align*}

By solving $\{f_1=0,\; f_2=0,\; f_3=0,\; f_4=0,\; f_5=0,\; f_6=0 \}$
for $u$, $v$, $w$, $x$, $y$, $z$,
we obtain $128$ solutions.
Note that $l\bar{m}\bar{n}=l-\bar{m}-\bar{n}$ is equivalent to
\begin{equation}
l = -\frac{\bar{m}+\bar{n}}{\bar{m}\bar{n}-1}.
\label{sol:exA.l.mn}
\end{equation}
By assigning each of the $128$ solutions together with \eqref{sol:exA.l.mn} to
$\{
\bar{m}y+\bar{n}z-2u-\bar{m}-\bar{n},\linebreak[0]\;
lx-\bar{n}z+2v-l+\bar{n},\linebreak[0]\;
lx-\bar{m}y+2w-l+\bar{m},\linebreak[0]\;
xy-w^2,\linebreak[0]\;
xz-v^2,\linebreak[0]\;
yz-u^2
\}$
and then picking out the solutions
such that the assignment makes all of the polynomials equal $0$,
we still have $8$ solutions
\eqref{lmn:a1}--\eqref{lmn:a8},
which are the solutions of~\eqref{eq:exA}.
\end{proof}

\nocite{Bottema2001,Stevanovic2003}
\bibliography{malfatti}
\suppressfloats[t]

\begin{figure}
\begin{center}
\includegraphics[width=\textwidth]{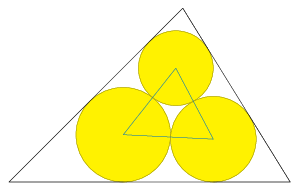}
\caption{%
\(
\protect\left\{
\protect\begin{aligned}
r_1 &=
\rA\sin^2(\sigma-\alpha)
,\protect\\
r_2 &=
\rB\sin^2(\sigma-\beta)
,\protect\\
r_3 &=
\rC\sin^2(\sigma-\gamma)
.
\protect\end{aligned}
\protect\right.
\)
}
\label{fig:i1}
\end{center}
\end{figure}
\begin{figure}
\begin{center}
\includegraphics[width=\textwidth]{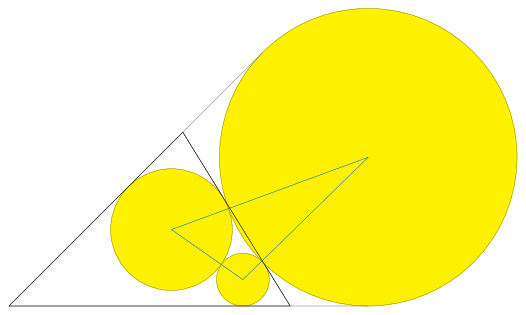}
\caption{%
\(
\protect\left\{
\protect\begin{aligned}
r_1 &=
\rA\sin^2\sigma
,\protect\\
r_2 &=
\rB\sin^2(\sigma-\gamma)
,\protect\\
r_3 &=
\rC\sin^2(\sigma-\beta)
.
\protect\end{aligned}
\protect\right.
\)
}
\end{center}
\end{figure}
\begin{figure}
\begin{center}
\includegraphics[width=\textwidth]{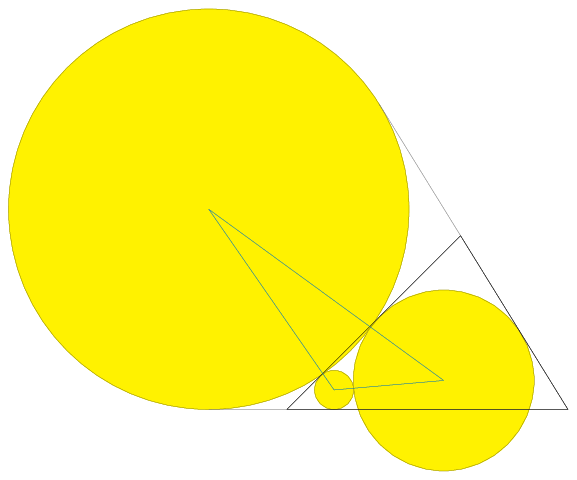}
\caption{%
\(
\protect\left\{
\protect\begin{aligned}
r_1 &=
\rA\sin^2(\sigma-\gamma)
,\protect\\
r_2 &=
\rB\sin^2\sigma
,\protect\\
r_3 &=
\rC\sin^2(\sigma-\alpha)
.
\protect\end{aligned}
\protect\right.
\)
}
\end{center}
\end{figure}
\begin{figure}
\begin{center}
\includegraphics[width=\textwidth]{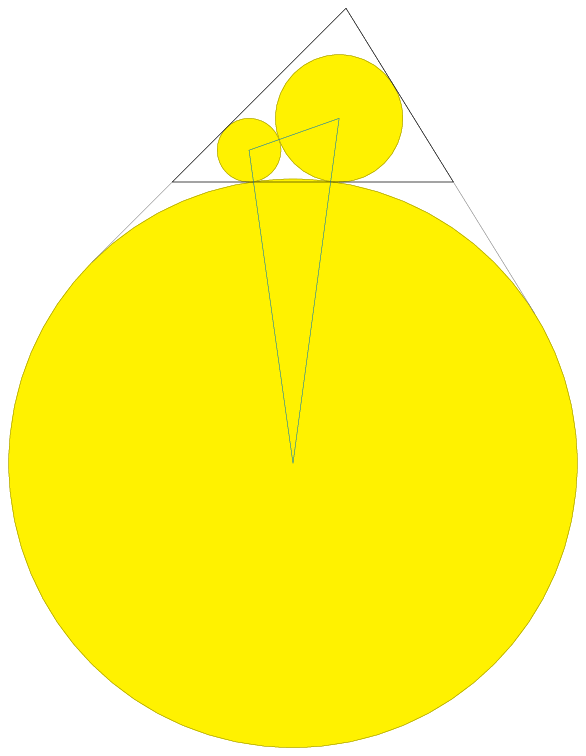}
\caption{%
\(
\protect\left\{
\protect\begin{aligned}
r_1 &=
\rA\sin^2(\sigma-\beta)
,\protect\\
r_2 &=
\rB\sin^2(\sigma-\alpha)
,\protect\\
r_3 &=
\rC\sin^2\sigma
.
\protect\end{aligned}
\protect\right.
\)
}
\end{center}
\end{figure}
\begin{figure}
\begin{center}
\includegraphics[width=\textwidth]{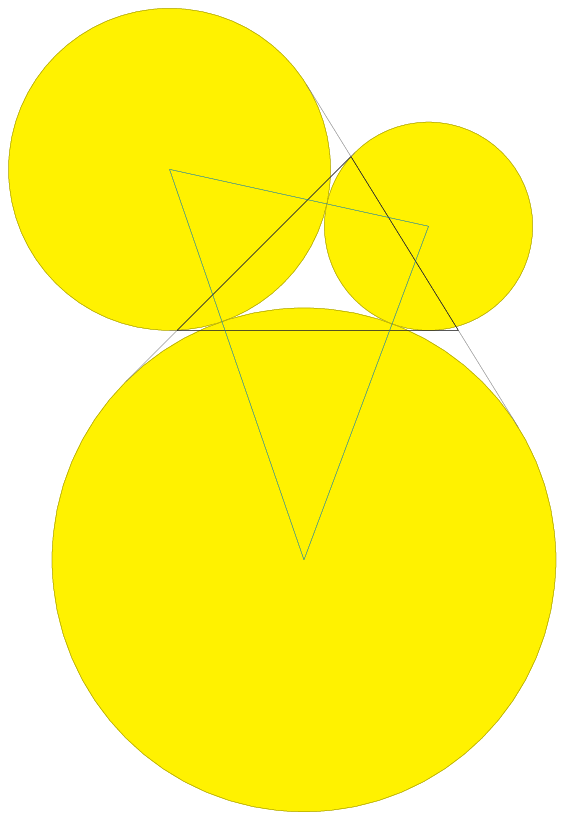}
\caption{%
\(
\protect\left\{
\protect\begin{aligned}
r_1 &=
\rA\cos^2(\sigma-\alpha)
,\protect\\
r_2 &=
\rB\cos^2(\sigma-\beta)
,\protect\\
r_3 &=
\rC\cos^2(\sigma-\gamma)
.
\protect\end{aligned}
\protect\right.
\)
}
\end{center}
\end{figure}
\begin{figure}
\begin{center}
\includegraphics[width=\textwidth]{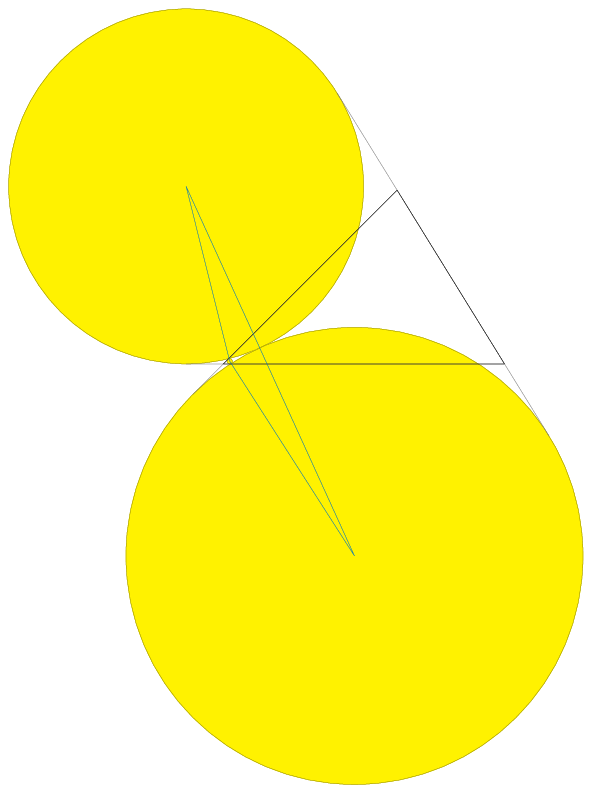}
\caption{%
\(
\protect\left\{
\protect\begin{aligned}
r_1 &=
\rA\cos^2\sigma
,\protect\\
r_2 &=
\rB\cos^2(\sigma-\gamma)
,\protect\\
r_3 &=
\rC\cos^2(\sigma-\beta)
.
\protect\end{aligned}
\protect\right.
\)
}
\end{center}
\end{figure}
\begin{figure}
\begin{center}
\includegraphics[width=\textwidth]{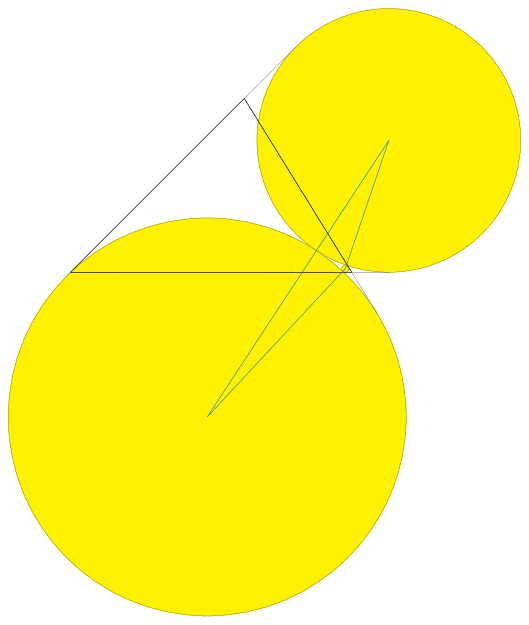}
\caption{%
\(
\protect\left\{
\protect\begin{aligned}
r_1 &=
\rA\cos^2(\sigma-\gamma)
,\protect\\
r_2 &=
\rB\cos^2\sigma
,\protect\\
r_3 &=
\rC\cos^2(\sigma-\alpha)
.
\protect\end{aligned}
\protect\right.
\)
}
\end{center}
\end{figure}
\begin{figure}
\begin{center}
\includegraphics[width=\textwidth]{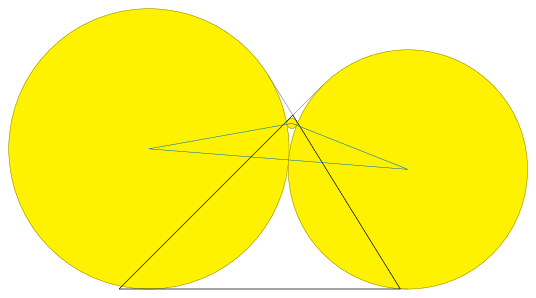}
\caption{%
\(
\protect\left\{
\protect\begin{aligned}
r_1 &=
\rA\cos^2(\sigma-\beta)
,\protect\\
r_2 &=
\rB\cos^2(\sigma-\alpha)
,\protect\\
r_3 &=
\rC\cos^2\sigma
.
\protect\end{aligned}
\protect\right.
\)
}
\end{center}
\end{figure}
\begin{figure}
\begin{center}
\includegraphics[width=\textwidth]{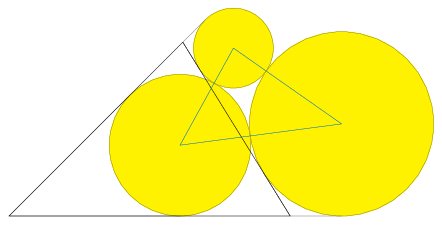}
\caption{%
\(
\protect\left\{
\protect\begin{aligned}
r_1 &=
r\cosh^2(\sigmaA-\alphaA)
,\protect\\
r_2 &=
\rC\sinh^2(\sigmaA-\betaA)
,\protect\\
r_3 &=
\rB\sinh^2(\sigmaA-\gammaA)
.
\protect\end{aligned}
\protect\right.
\)
}
\end{center}
\end{figure}
\begin{figure}
\begin{center}
\includegraphics[width=\textwidth]{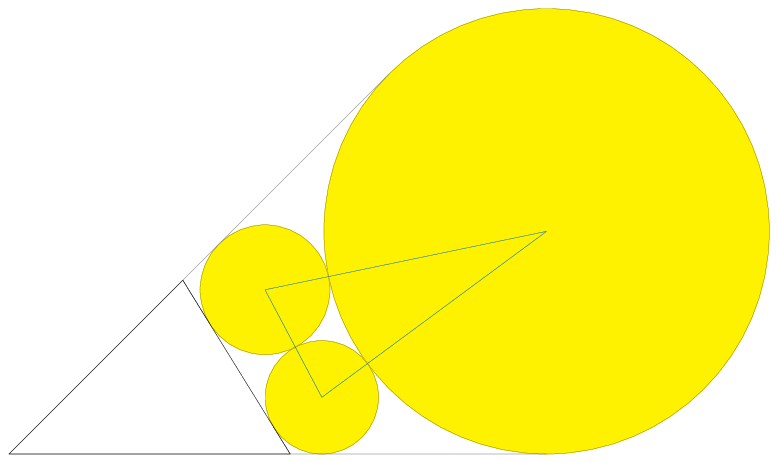}
\caption{%
\(
\protect\left\{
\protect\begin{aligned}
r_1 &=
r\cosh^2\sigmaA
,\protect\\
r_2 &=
\rC\sinh^2(\sigmaA-\gammaA)
,\protect\\
r_3 &=
\rB\sinh^2(\sigmaA-\betaA)
.
\protect\end{aligned}
\protect\right.
\)
}
\end{center}
\end{figure}
\begin{figure}
\begin{center}
\includegraphics[width=\textwidth]{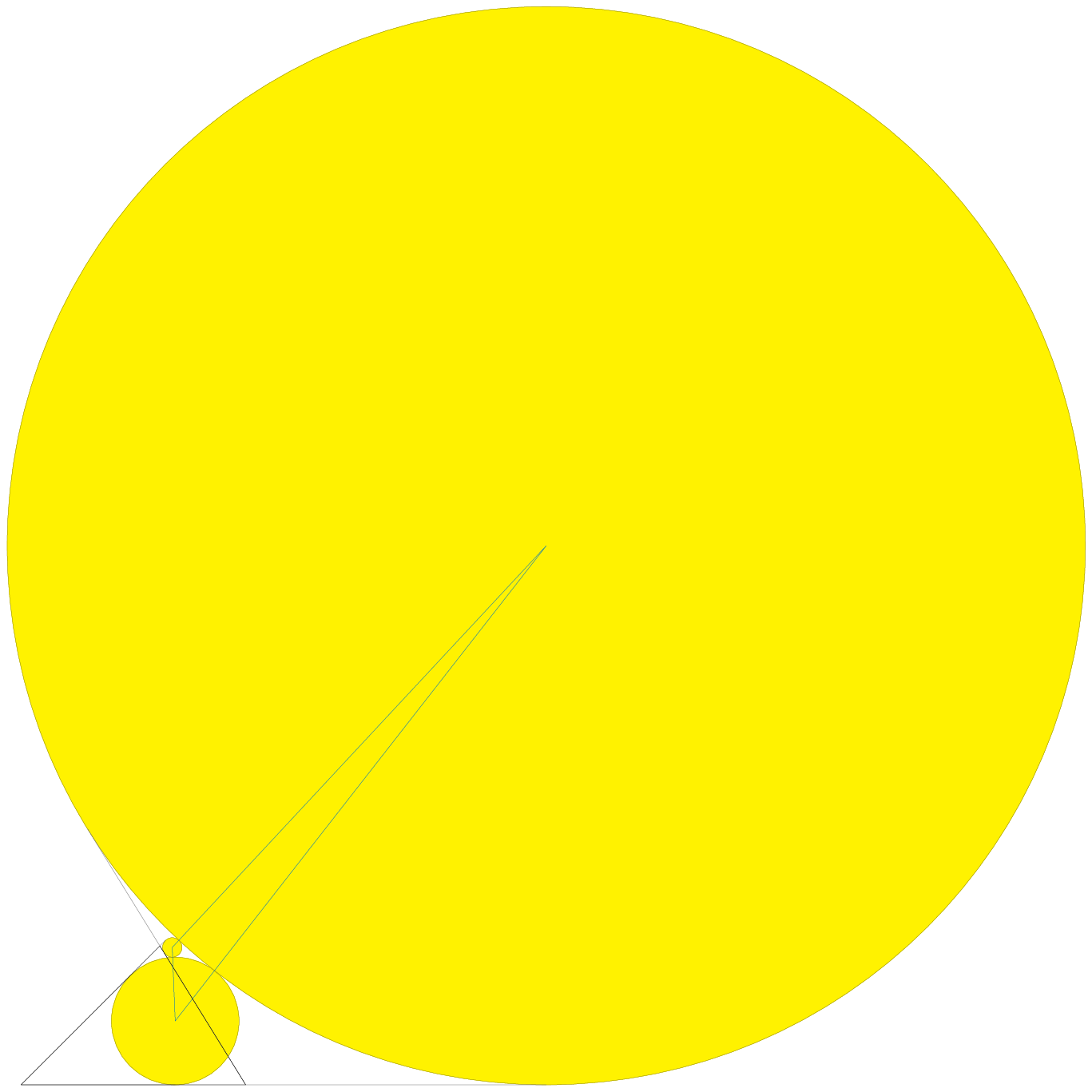}
\caption{%
\(
\protect\left\{
\protect\begin{aligned}
r_1 &=
r\cosh^2(\sigmaA-\gammaA)
,\protect\\
r_2 &=
\rC\sinh^2\sigmaA
,\protect\\
r_3 &=
\rB\sinh^2(\sigmaA-\alphaA)
.
\protect\end{aligned}
\protect\right.
\)
}
\end{center}
\end{figure}
\begin{figure}
\begin{center}
\includegraphics[width=\textwidth]{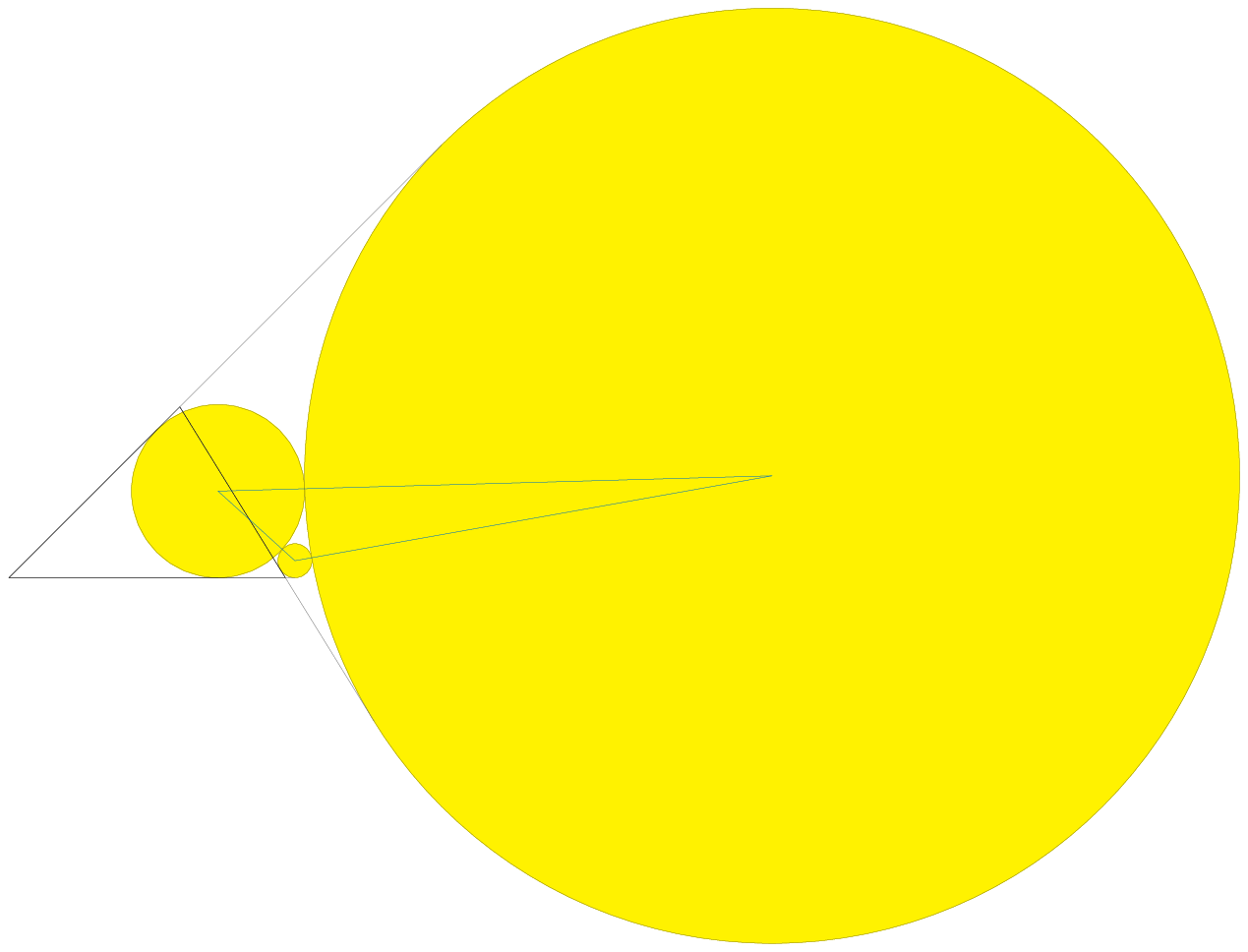}
\caption{%
\(
\protect\left\{
\protect\begin{aligned}
r_1 &=
r\cosh^2(\sigmaA-\betaA)
,\protect\\
r_2 &=
\rC\sinh^2(\sigmaA-\alphaA)
,\protect\\
r_3 &=
\rB\sinh^2\sigmaA
.
\protect\end{aligned}
\protect\right.
\)
}
\end{center}
\end{figure}
\begin{figure}
\begin{center}
\includegraphics[width=\textwidth]{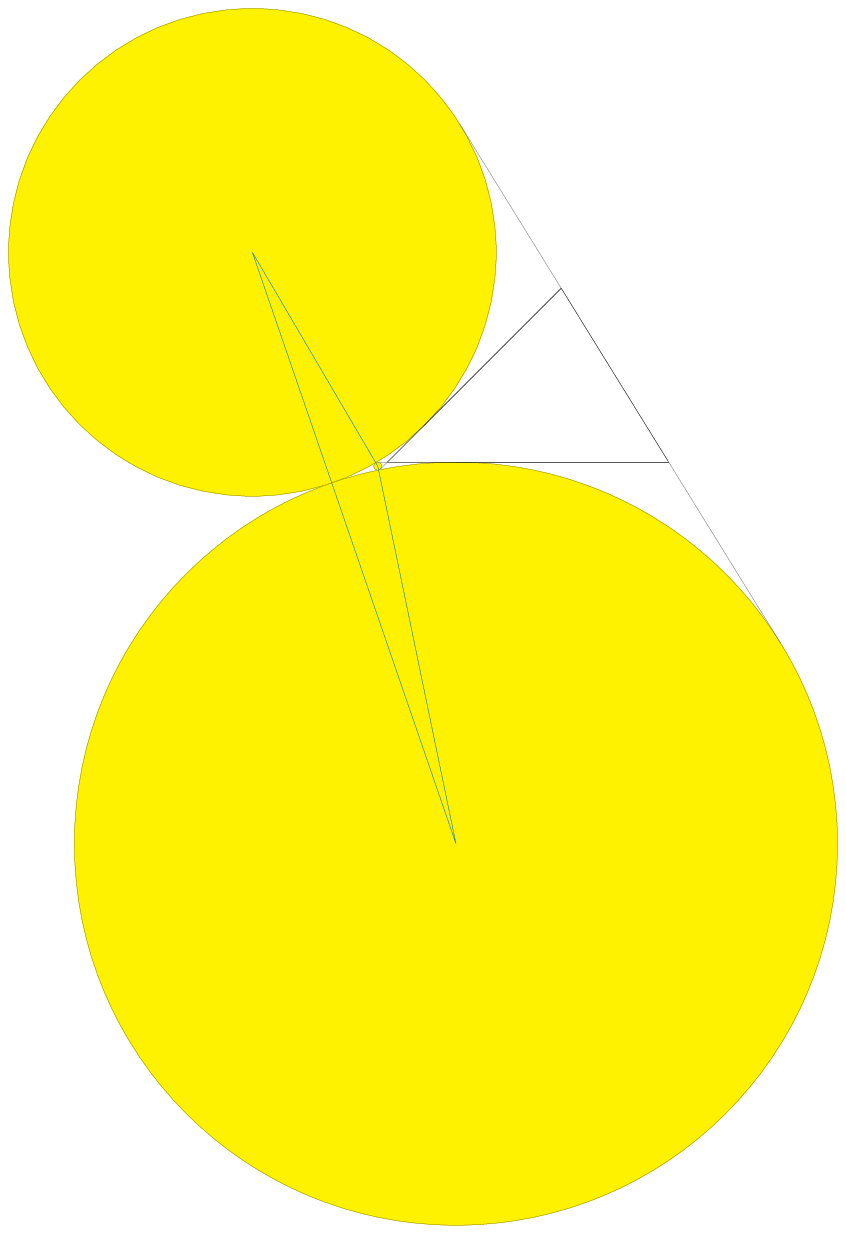}
\caption{%
\(
\protect\left\{
\protect\begin{aligned}
r_1 &=
r\sinh^2(\sigmaA-\alphaA)
,\protect\\
r_2 &=
\rC\cosh^2(\sigmaA-\betaA)
,\protect\\
r_3 &=
\rB\cosh^2(\sigmaA-\gammaA)
.
\protect\end{aligned}
\protect\right.
\)
}
\end{center}
\end{figure}
\begin{figure}
\begin{center}
\includegraphics[width=\textwidth]{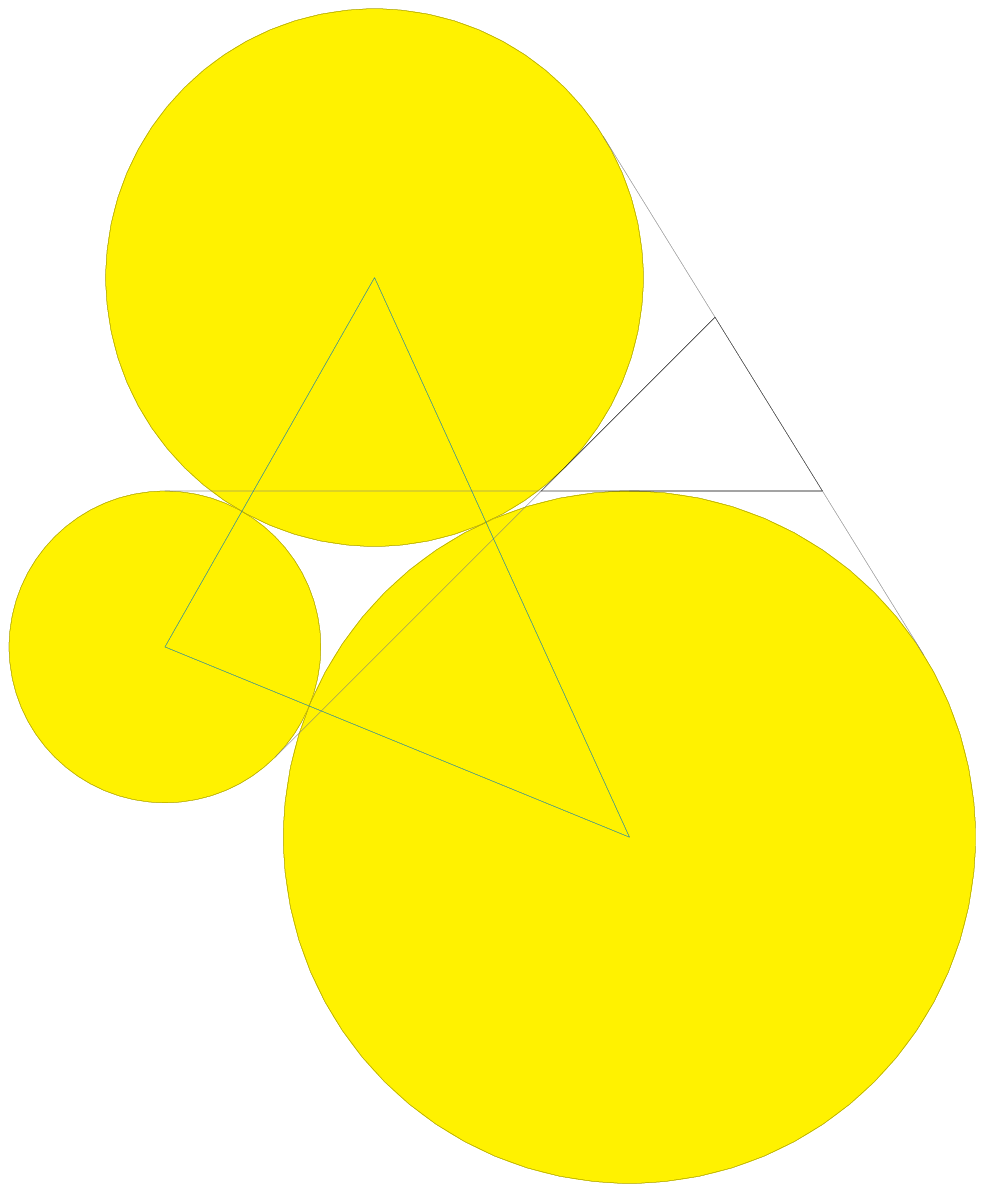}
\caption{%
\(
\protect\left\{
\protect\begin{aligned}
r_1 &=
r\sinh^2\sigmaA
,\protect\\
r_2 &=
\rC\cosh^2(\sigmaA-\gammaA)
,\protect\\
r_3 &=
\rB\cosh^2(\sigmaA-\betaA)
.
\protect\end{aligned}
\protect\right.
\)
}
\end{center}
\end{figure}
\begin{figure}
\begin{center}
\includegraphics[width=\textwidth]{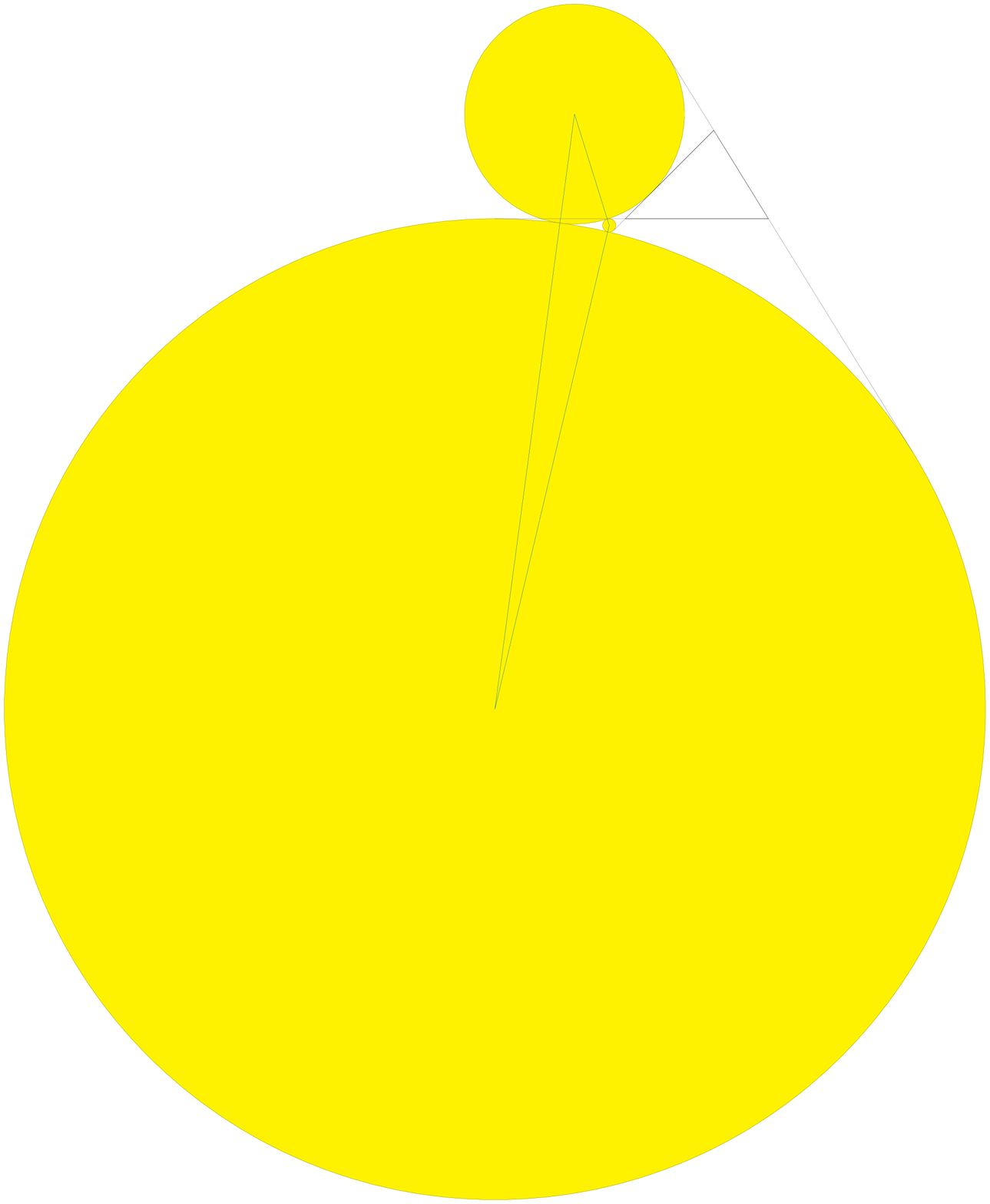}
\caption{%
\(
\protect\left\{
\protect\begin{aligned}
r_1 &=
r\sinh^2(\sigmaA-\gammaA)
,\protect\\
r_2 &=
\rC\cosh^2\sigmaA
,\protect\\
r_3 &=
\rB\cosh^2(\sigmaA-\alphaA)
.
\protect\end{aligned}
\protect\right.
\)
}
\end{center}
\end{figure}
\begin{figure}
\begin{center}
\includegraphics[width=\textwidth]{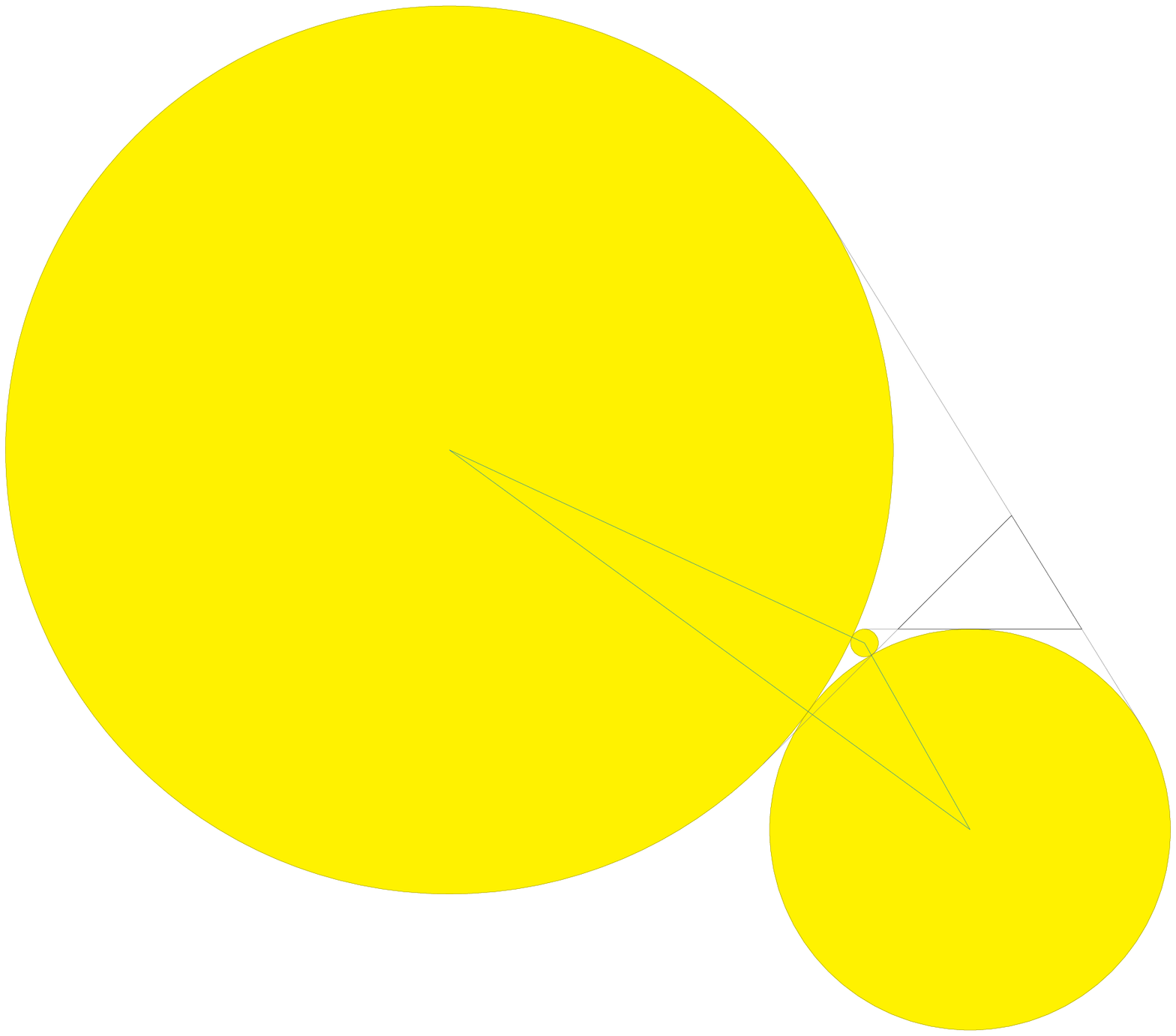}
\caption{%
\(
\protect\left\{
\protect\begin{aligned}
r_1 &=
r\sinh^2(\sigmaA-\betaA)
,\protect\\
r_2 &=
\rC\cosh^2(\sigmaA-\alphaA)
,\protect\\
r_3 &=
\rB\cosh^2\sigmaA
.
\protect\end{aligned}
\protect\right.
\)
}
\end{center}
\end{figure}
\begin{figure}
\begin{center}
\includegraphics[width=\textwidth]{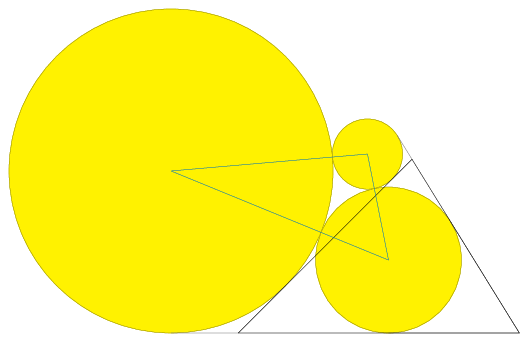}
\caption{%
\(
\protect\left\{
\protect\begin{aligned}
r_1 &=
\rC\sinh^2(\sigmaB-\alphaB)
,\protect\\
r_2 &=
r\cosh^2(\sigmaB-\betaB)
,\protect\\
r_3 &=
\rA\sinh^2(\sigmaB-\gammaB)
.
\protect\end{aligned}
\protect\right.
\)
}
\end{center}
\end{figure}
\clearpage
\begin{figure}
\begin{center}
\includegraphics[width=\textwidth]{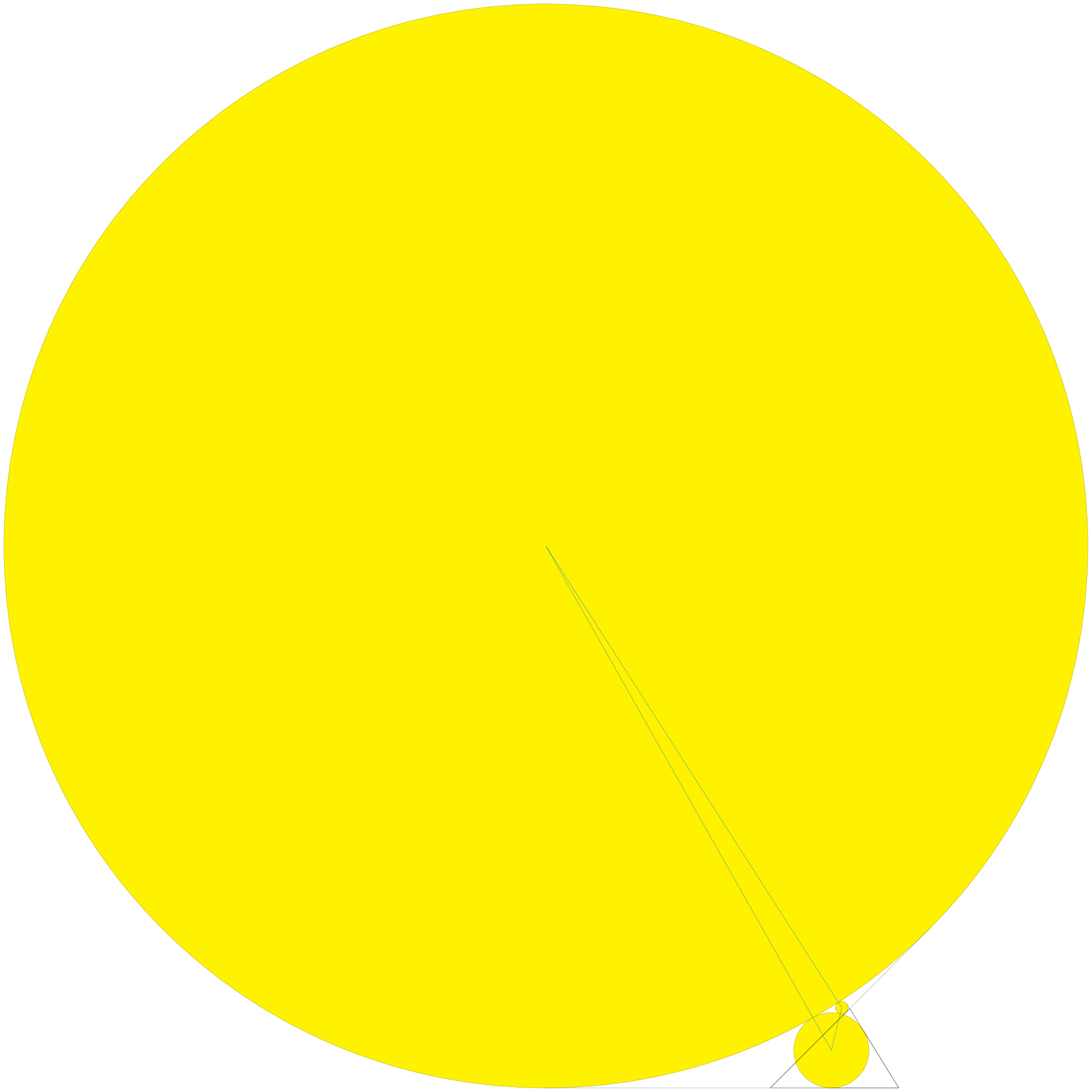}
\caption{%
\(
\protect\left\{
\protect\begin{aligned}
r_1 &=
\rC\sinh^2\sigmaB
,\protect\\
r_2 &=
r\cosh^2(\sigmaB-\gammaB)
,\protect\\
r_3 &=
\rA\sinh^2(\sigmaB-\betaB)
.
\protect\end{aligned}
\protect\right.
\)
}
\end{center}
\end{figure}
\begin{figure}
\begin{center}
\includegraphics[width=\textwidth]{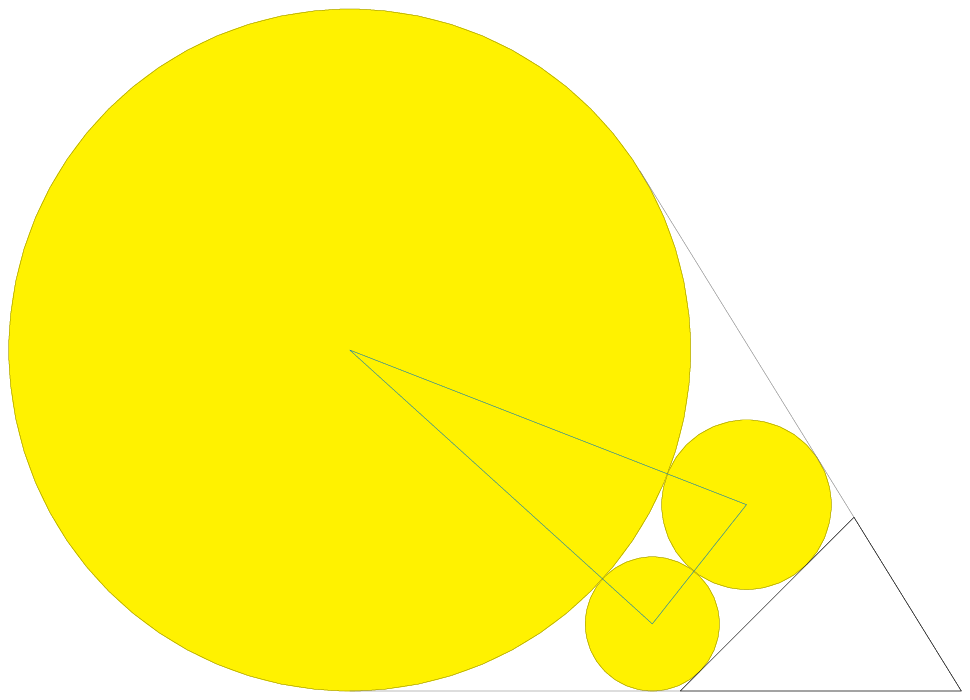}
\caption{%
\(
\protect\left\{
\protect\begin{aligned}
r_1 &=
\rC\sinh^2(\sigmaB-\gammaB)
,\protect\\
r_2 &=
r\cosh^2\sigmaB
,\protect\\
r_3 &=
\rA\sinh^2(\sigmaB-\alphaB)
.
\protect\end{aligned}
\protect\right.
\)
}
\end{center}
\end{figure}
\begin{figure}
\begin{center}
\includegraphics[width=\textwidth]{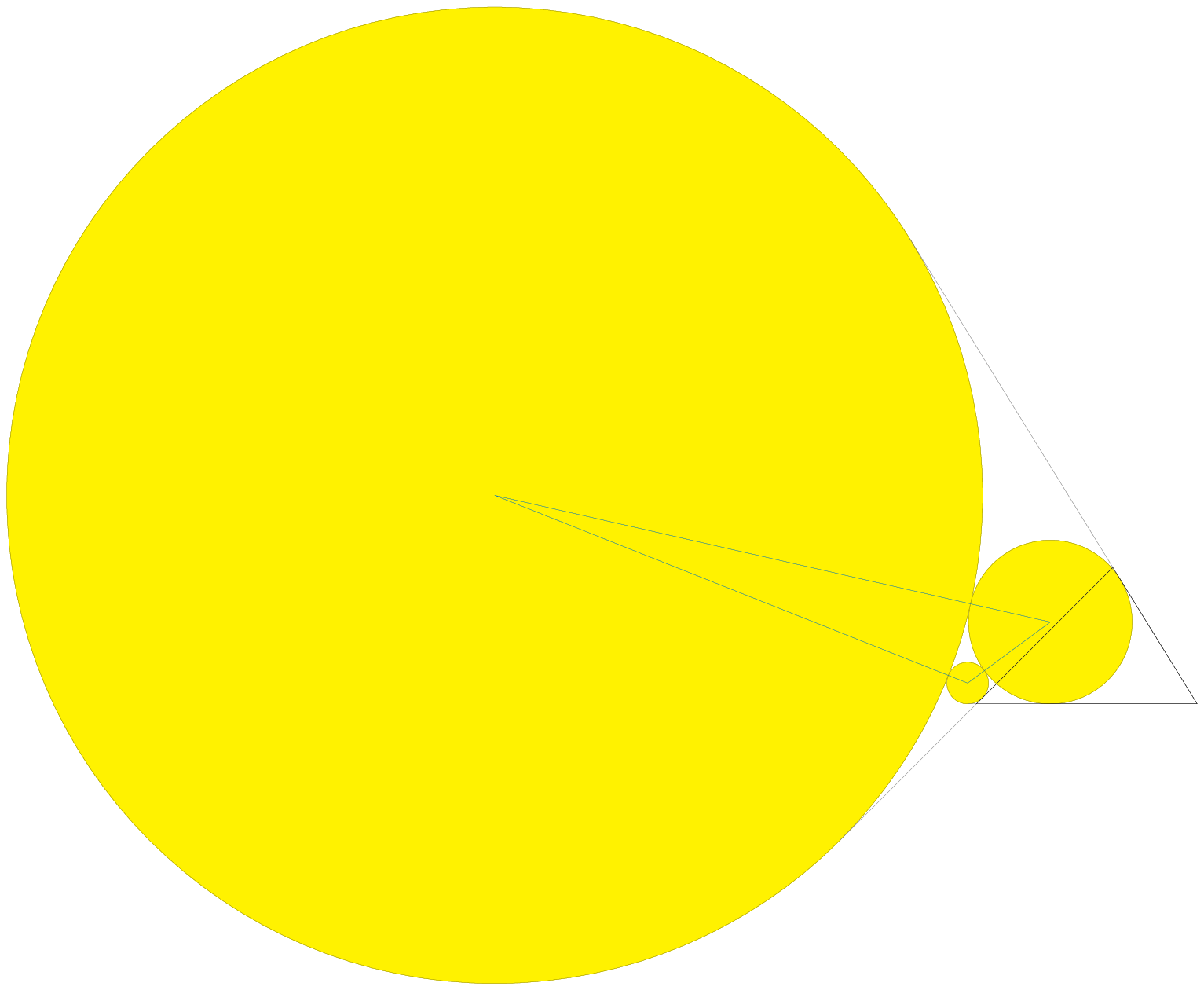}
\caption{%
\(
\protect\left\{
\protect\begin{aligned}
r_1 &=
\rC\sinh^2(\sigmaB-\betaB)
,\protect\\
r_2 &=
r\cosh^2(\sigmaB-\alphaB)
,\protect\\
r_3 &=
\rA\sinh^2\sigmaB
.
\protect\end{aligned}
\protect\right.
\)
}
\end{center}
\end{figure}
\begin{figure}
\begin{center}
\includegraphics[width=\textwidth]{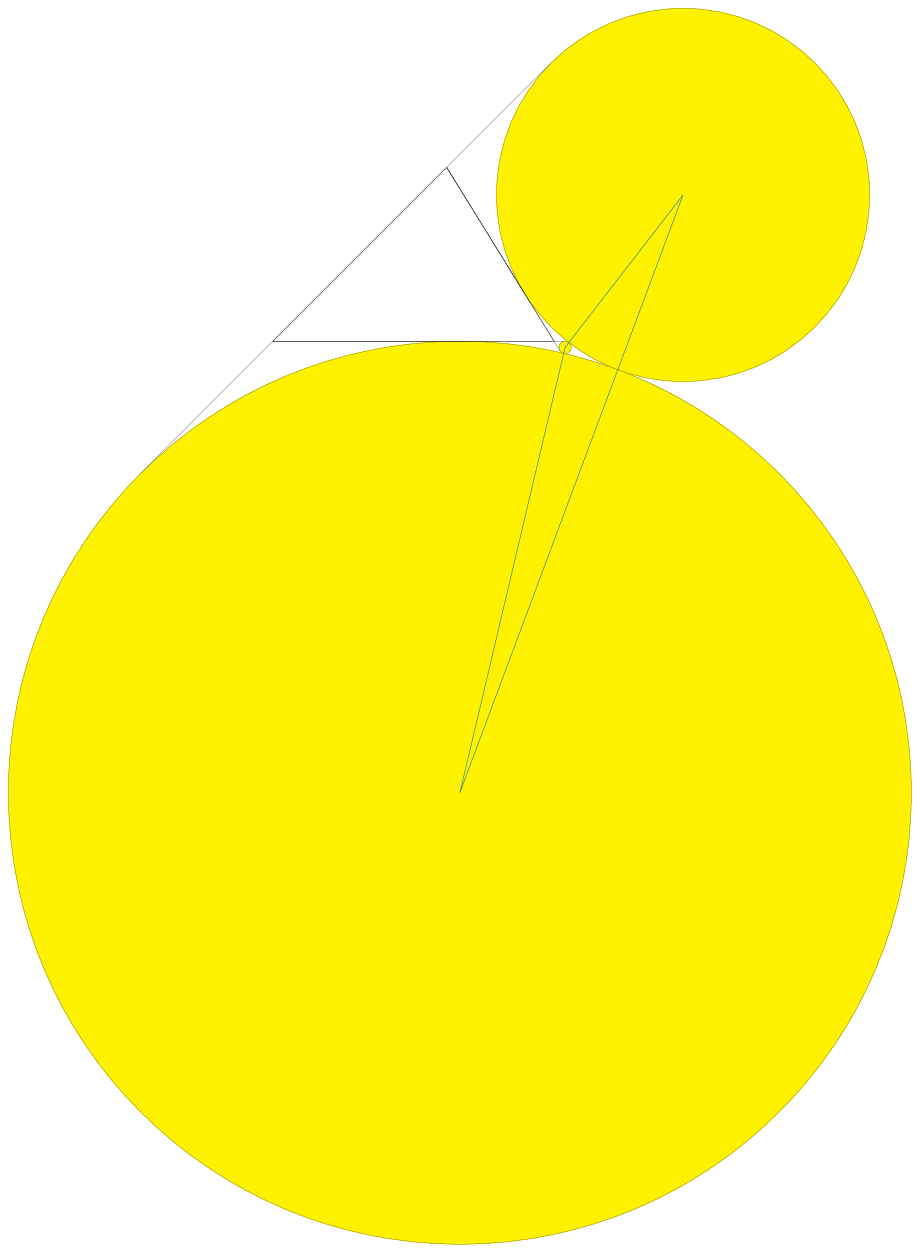}
\caption{%
\(
\protect\left\{
\protect\begin{aligned}
r_1 &=
\rC\cosh^2(\sigmaB-\alphaB)
,\protect\\
r_2 &=
r\sinh^2(\sigmaB-\betaB)
,\protect\\
r_3 &=
\rA\cosh^2(\sigmaB-\gammaB)
.
\protect\end{aligned}
\protect\right.
\)
}
\end{center}
\end{figure}
\begin{figure}
\begin{center}
\includegraphics[width=\textwidth]{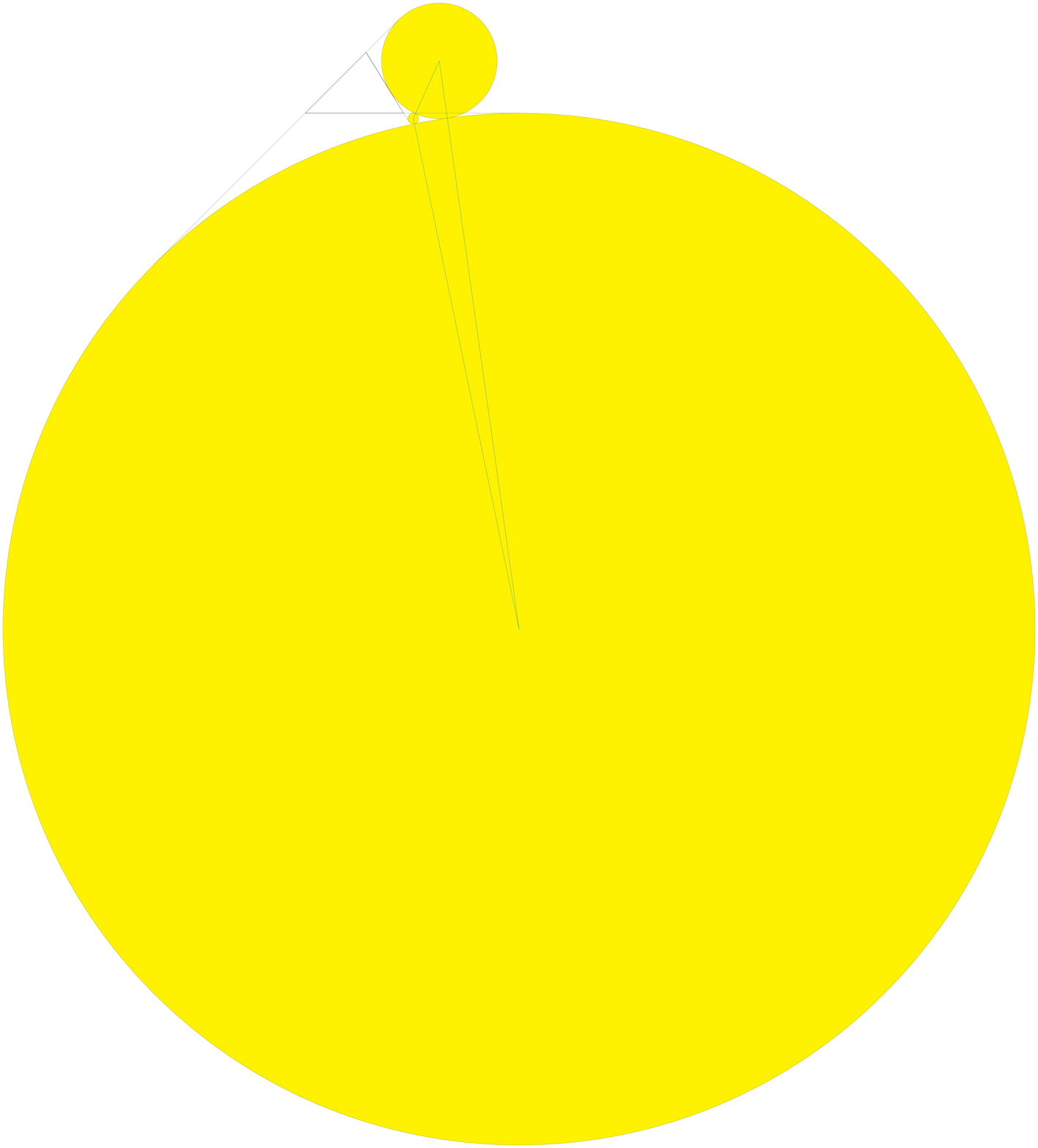}
\caption{%
\(
\protect\left\{
\protect\begin{aligned}
r_1 &=
\rC\cosh^2\sigmaB
,\protect\\
r_2 &=
r\sinh^2(\sigmaB-\gammaB)
,\protect\\
r_3 &=
\rB\cosh^2(\sigmaB-\betaB)
.
\protect\end{aligned}
\protect\right.
\)
}
\end{center}
\end{figure}
\begin{figure}
\begin{center}
\includegraphics[width=\textwidth]{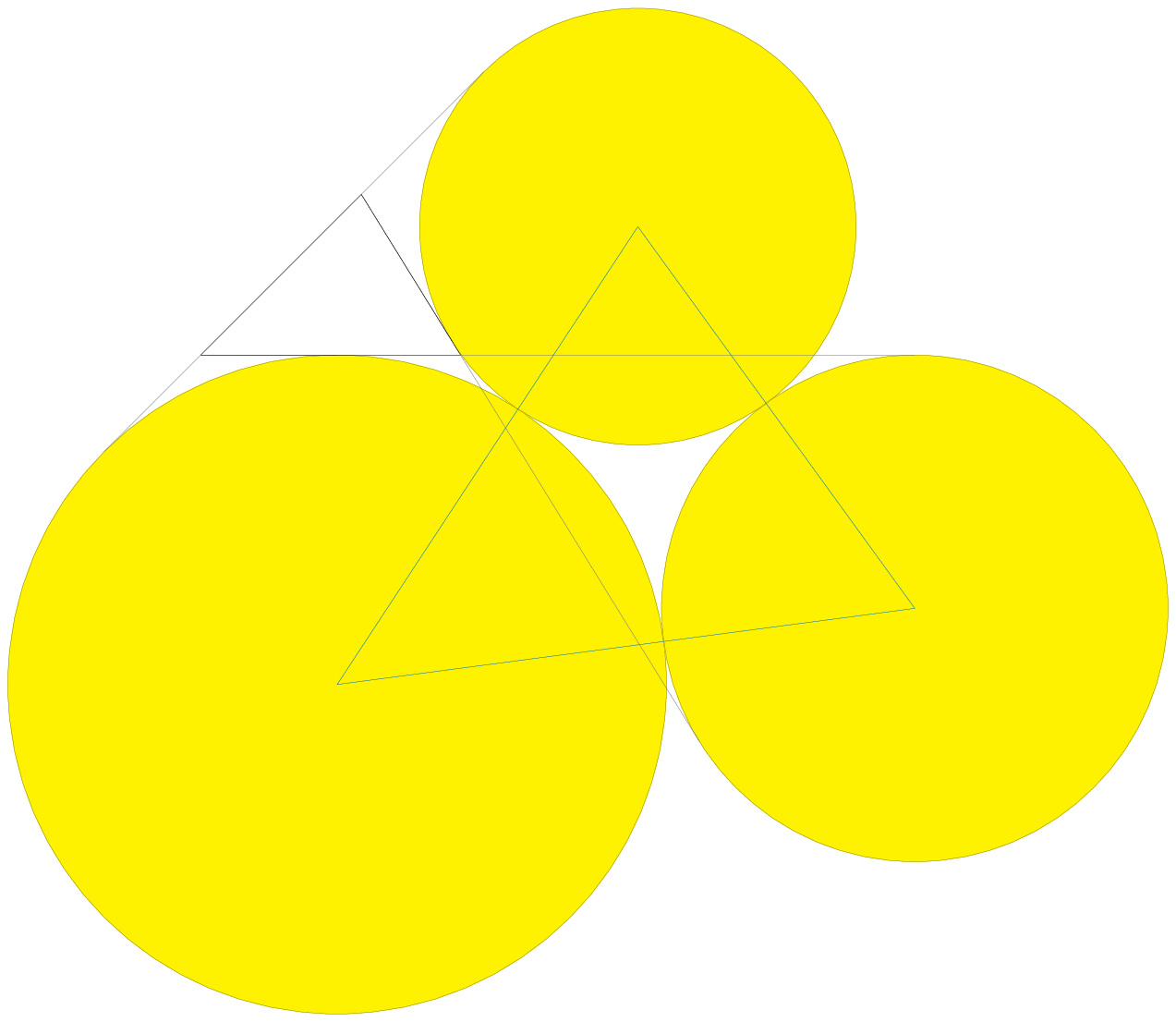}
\caption{%
\(
\protect\left\{
\protect\begin{aligned}
r_1 &=
\rC\cosh^2(\sigmaB-\gammaB)
,\protect\\
r_2 &=
r\sinh^2\sigmaB
,\protect\\
r_3 &=
\rA\cosh^2(\sigmaB-\alphaB)
.
\protect\end{aligned}
\protect\right.
\)
}
\end{center}
\end{figure}
\begin{figure}
\begin{center}
\includegraphics[width=\textwidth]{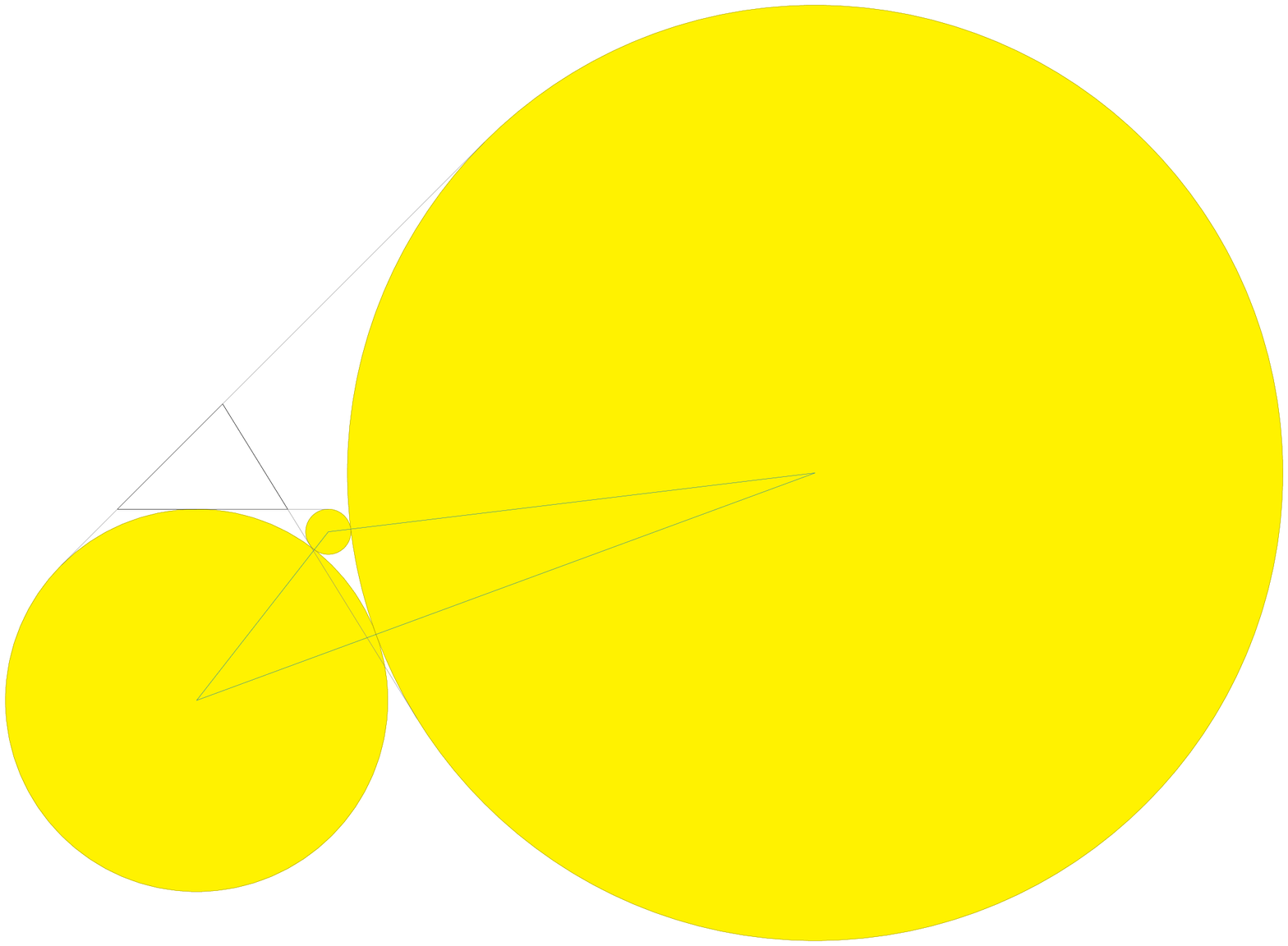}
\caption{%
\(
\protect\left\{
\protect\begin{aligned}
r_1 &=
\rC\cosh^2(\sigmaB-\betaB)
,\protect\\
r_2 &=
r\sinh^2(\sigmaB-\alphaB)
,\protect\\
r_3 &=
\rA\cosh^2\sigmaB
.
\protect\end{aligned}
\protect\right.
\)
}
\end{center}
\end{figure}
\begin{figure}
\begin{center}
\includegraphics[width=\textwidth]{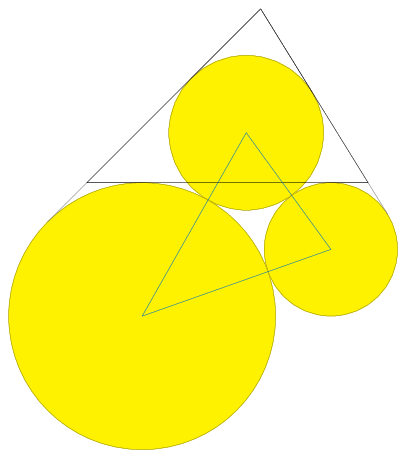}
\caption{%
\(
\protect\left\{
\protect\begin{aligned}
r_1 &=
\rB\sinh^2(\sigmaC-\alphaC)
,\protect\\
r_2 &=
\rA\sinh^2(\sigmaC-\betaC)
,\protect\\
r_3 &=
r\cosh^2(\sigmaC-\gammaC)
.
\protect\end{aligned}
\protect\right.
\)
}
\end{center}
\end{figure}
\begin{figure}
\begin{center}
\includegraphics[width=\textwidth]{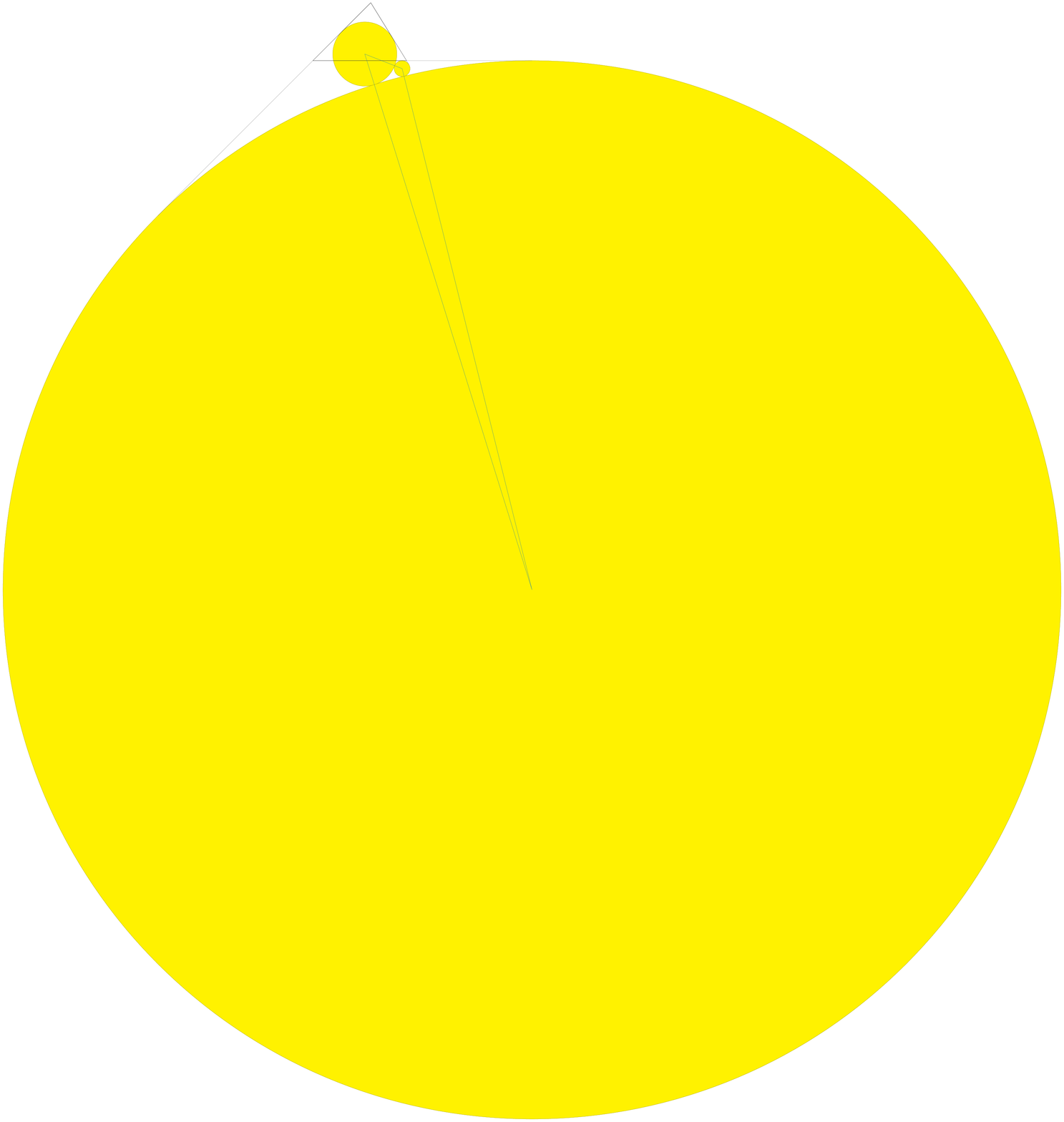}
\caption{%
\(
\protect\left\{
\protect\begin{aligned}
r_1 &=
\rB\sinh^2\sigmaC
,\protect\\
r_2 &=
\rA\sinh^2(\sigmaC-\gammaC)
,\protect\\
r_3 &=
r\cosh^2(\sigmaC-\betaC)
.
\protect\end{aligned}
\protect\right.
\)
}
\end{center}
\end{figure}
\begin{figure}
\begin{center}
\includegraphics[width=\textwidth]{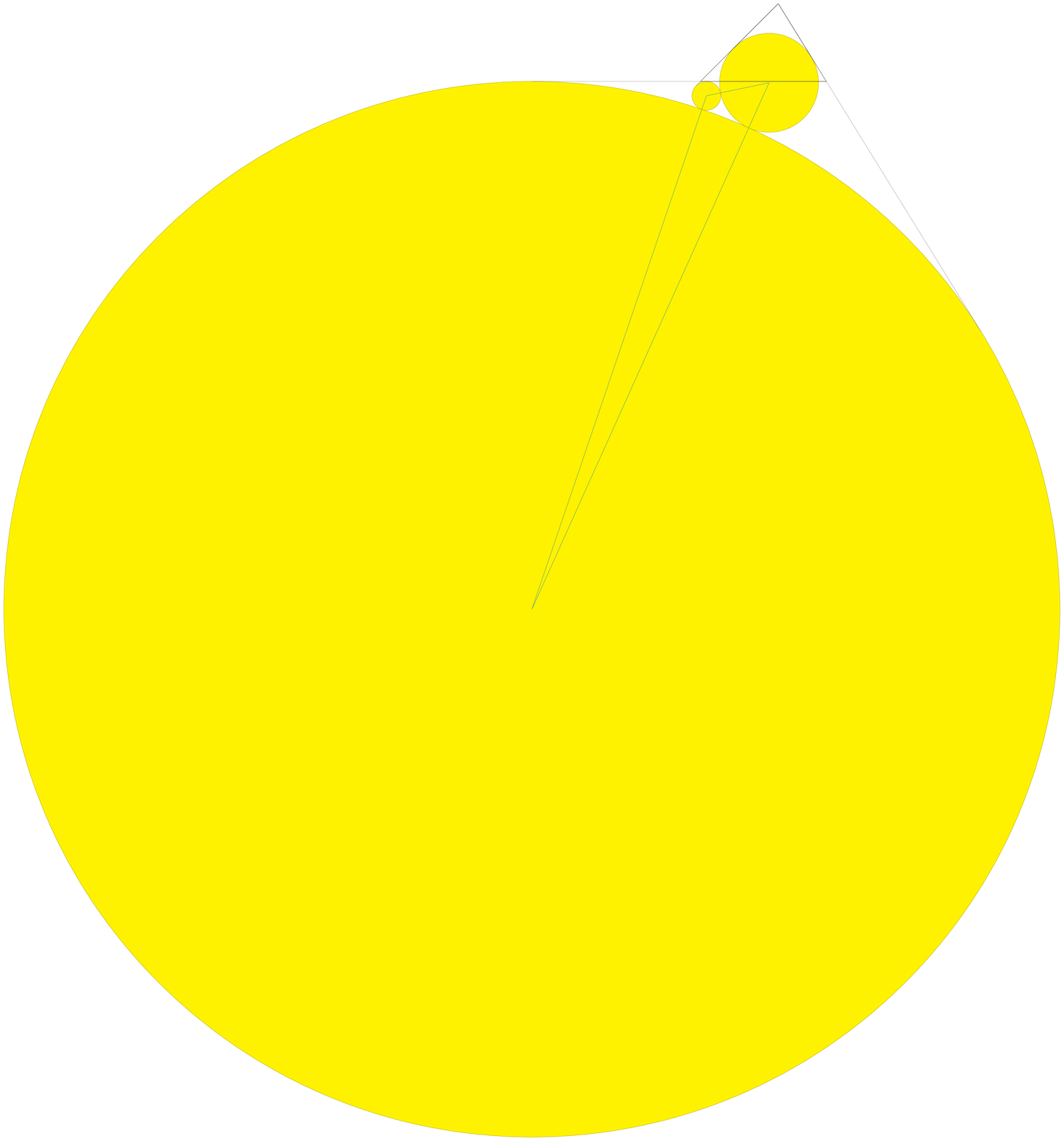}
\caption{%
\(
\protect\left\{
\protect\begin{aligned}
r_1 &=
\rB\sinh^2(\sigmaC-\gammaC)
,\protect\\
r_2 &=
\rA\sinh^2\sigmaC
,\protect\\
r_3 &=
r\cosh^2(\sigmaC-\alphaC)
.
\protect\end{aligned}
\protect\right.
\)
}
\end{center}
\end{figure}
\begin{figure}
\begin{center}
\includegraphics[width=\textwidth]{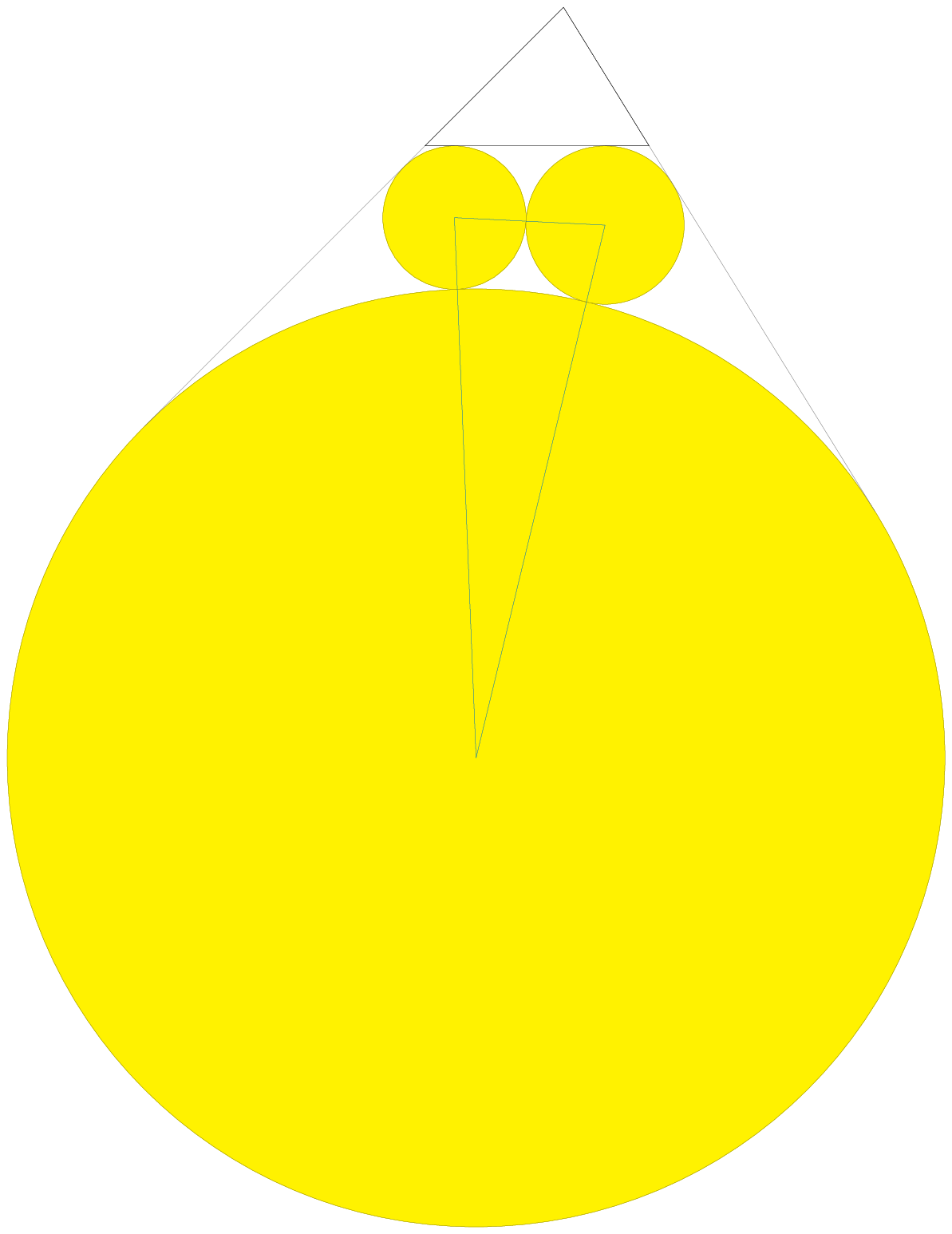}
\caption{%
\(
\protect\left\{
\protect\begin{aligned}
r_1 &=
\rB\sinh^2(\sigmaC-\betaC)
,\protect\\
r_2 &=
\rA\sinh^2(\sigmaC-\alphaC)
,\protect\\
r_3 &=
r\cosh^2\sigmaC
.
\protect\end{aligned}
\protect\right.
\)
}
\end{center}
\end{figure}
\begin{figure}
\begin{center}
\includegraphics[width=\textwidth]{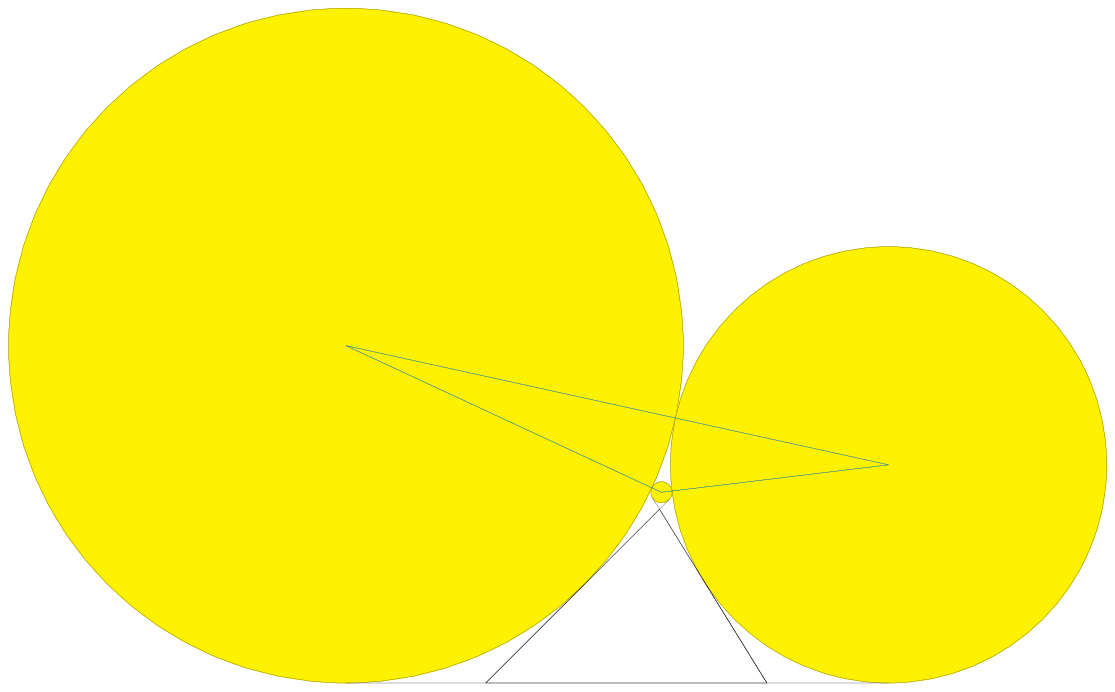}
\caption{%
\(
\protect\left\{
\protect\begin{aligned}
r_1 &=
\rB\cosh^2(\sigmaC-\alphaC)
,\protect\\
r_2 &=
\rA\cosh^2(\sigmaC-\betaC)
,\protect\\
r_3 &=
r\sinh^2(\sigmaC-\gammaC)
.
\protect\end{aligned}
\protect\right.
\)
}
\end{center}
\end{figure}
\begin{figure}
\begin{center}
\includegraphics[width=\textwidth]{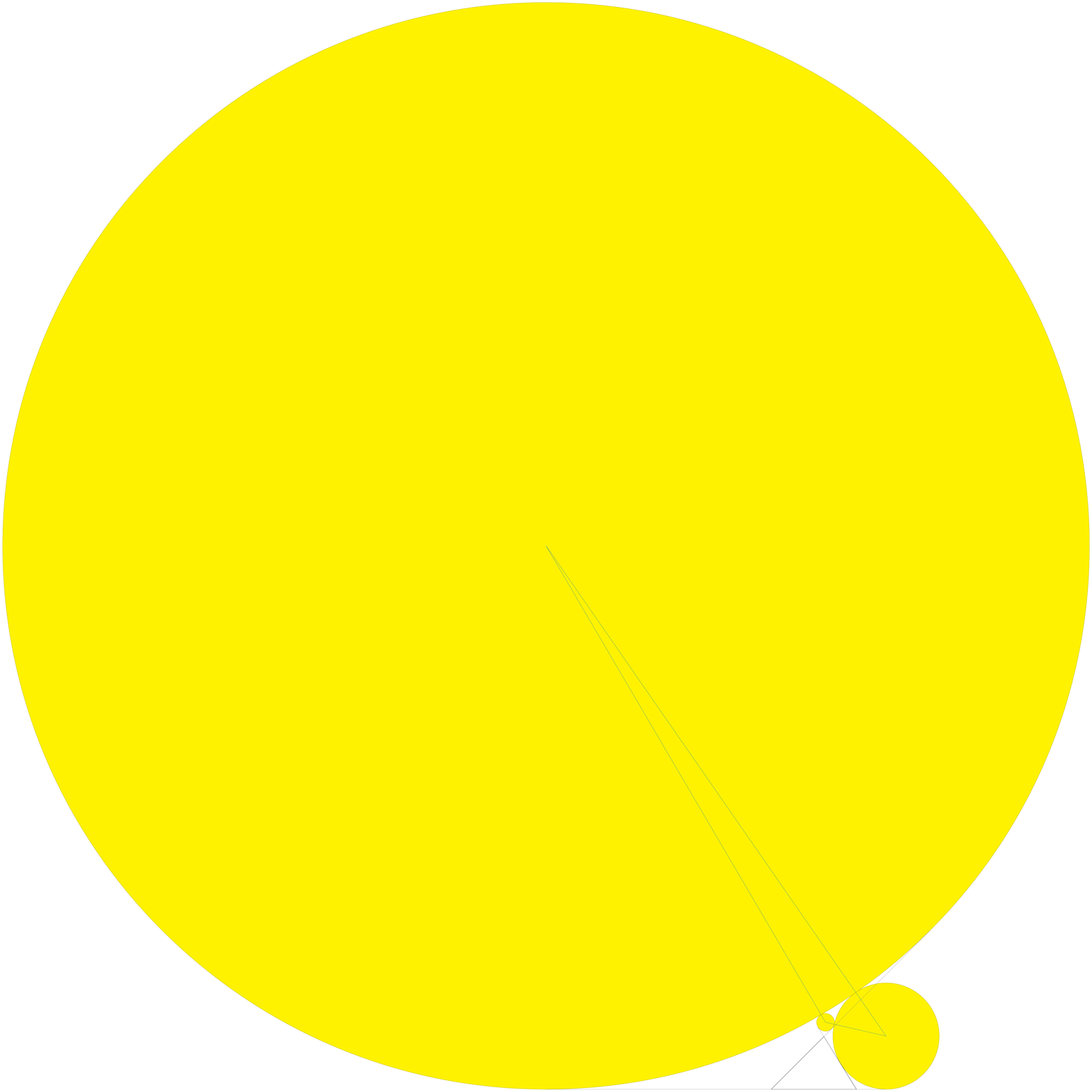}
\caption{%
\(
\protect\left\{
\protect\begin{aligned}
r_1 &=
\rB\cosh^2\sigmaC
,\protect\\
r_2 &=
\rA\cosh^2(\sigmaC-\gammaC)
,\protect\\
r_3 &=
r\sinh^2(\sigmaC-\betaC)
.
\protect\end{aligned}
\protect\right.
\)
}
\end{center}
\end{figure}
\begin{figure}
\begin{center}
\includegraphics[width=\textwidth]{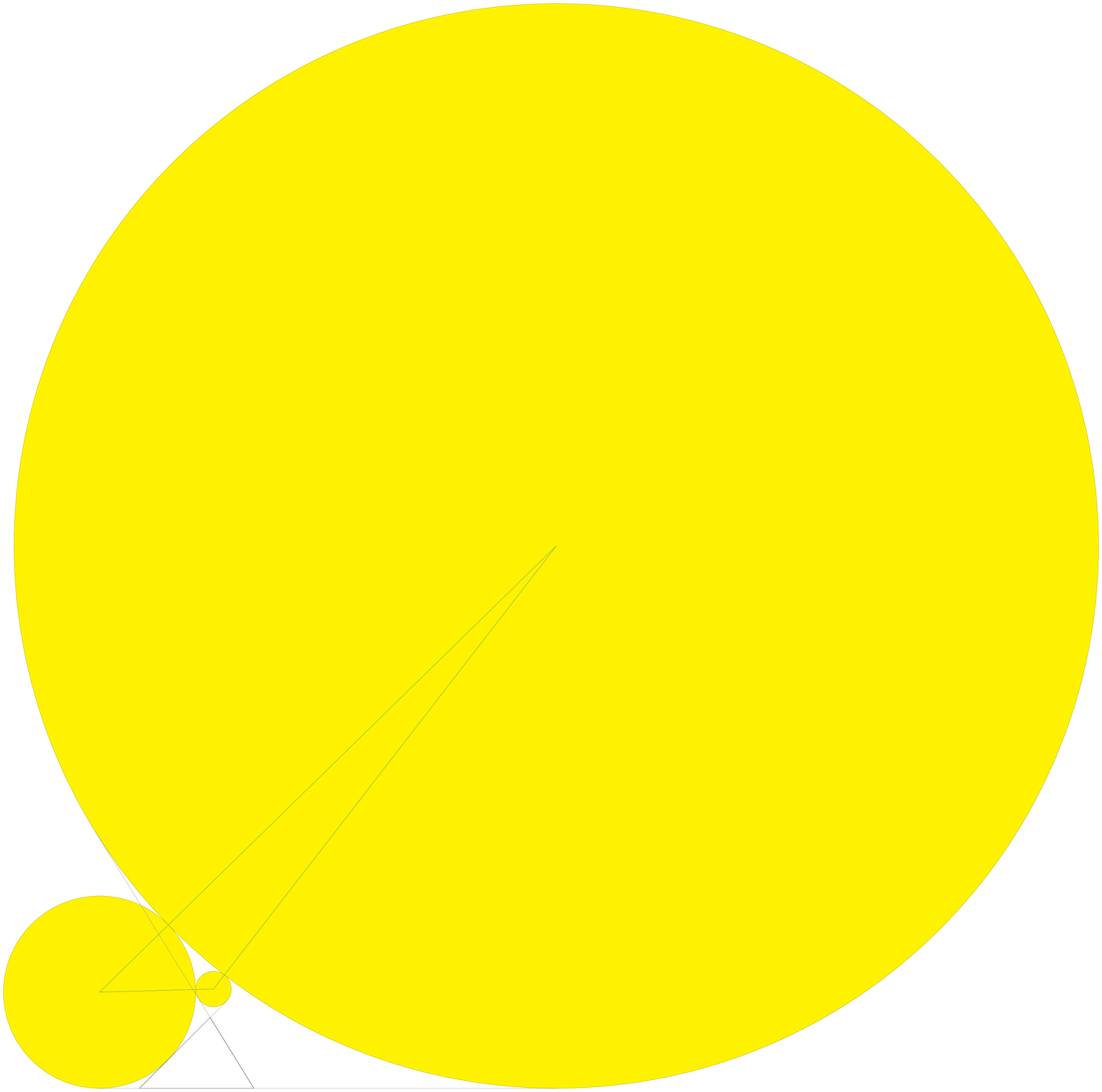}
\caption{%
\(
\protect\left\{
\protect\begin{aligned}
r_1 &=
\rB\cosh^2(\sigmaC-\gammaC)
,\protect\\
r_2 &=
\rA\cosh^2\sigmaC
,\protect\\
r_3 &=
r\sinh^2(\sigmaC-\betaC)
.
\protect\end{aligned}
\protect\right.
\)
}
\end{center}
\end{figure}
\begin{figure}
\begin{center}
\includegraphics[width=\textwidth]{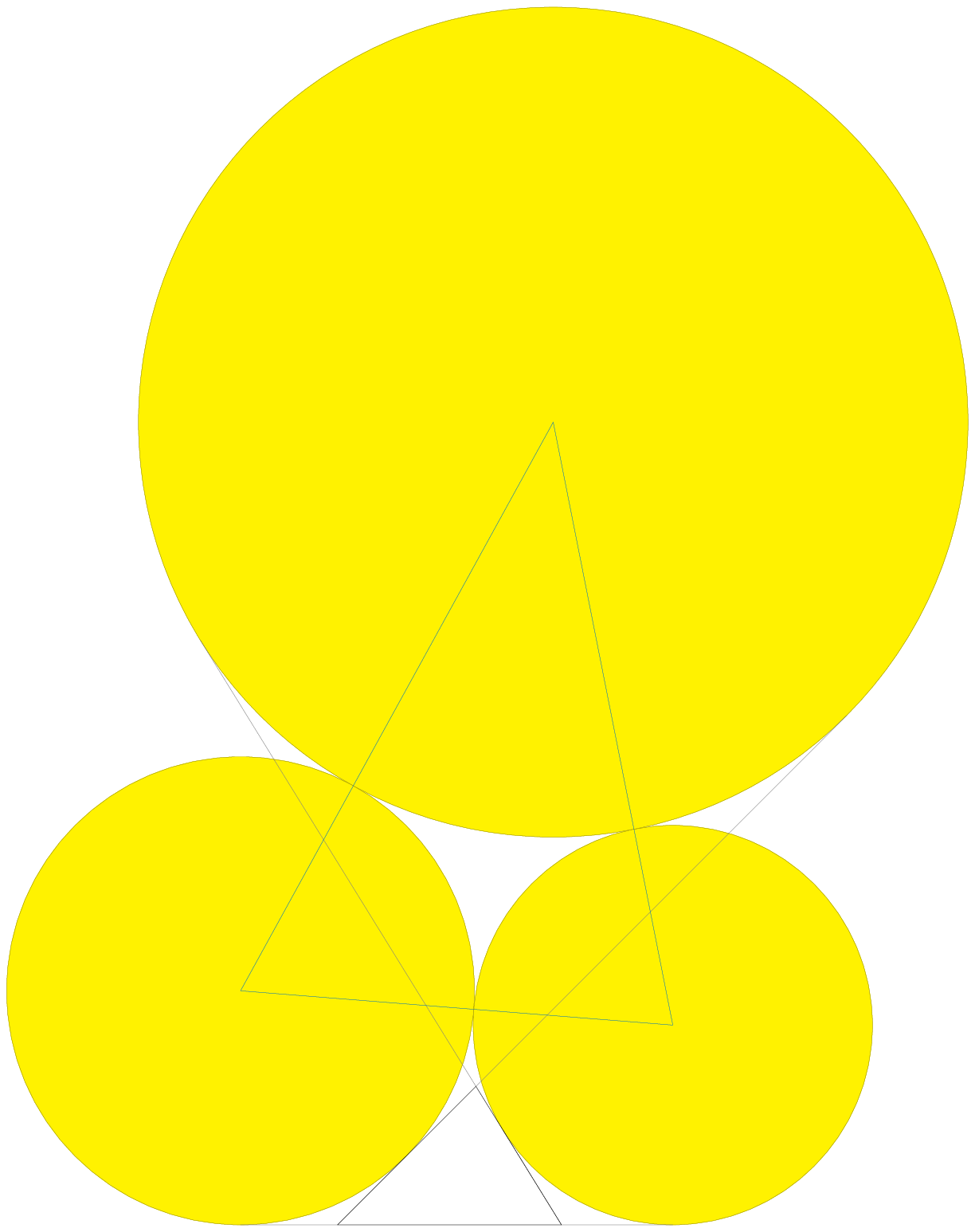}
\caption{%
\(
\protect\left\{
\protect\begin{aligned}
r_1 &=
\rB\cosh^2(\sigmaC-\betaC)
,\protect\\
r_2 &=
\rA\cosh^2(\sigmaC-\alphaC)
,\protect\\
r_3 &=
r\sinh^2\sigmaC
.
\protect\end{aligned}
\protect\right.
\)
}
\label{fig:c8}
\end{center}
\end{figure}

\end{document}